\documentclass[preprint,12pt]{elsarticle}

\usepackage{amssymb}
\usepackage{amssymb, amsmath, mathrsfs}
\usepackage{latexsym}
\usepackage{amssymb}
\usepackage{amsmath}
\usepackage{amsthm}
\usepackage{indentfirst}
\usepackage{color}

\allowdisplaybreaks

\newtheorem{theorem}{\indent Theorem}[section]
\newtheorem{lemma}{\indent Lemma}[section]\rm
\newtheorem{proposition}{\indent
Proposition}[section]\rm
\rm
\rm
\rm
\rm

\journal{ }

\begin{document}

\begin{frontmatter}



\title{The motion of closed hypersurfaces in the central force fields}


\author{Weiping Yan\corref{cor1}}
\cortext[cor1]{Corresponding author} \ead{mathyanwp@126.com;yanwp@xmu.edu.cn}
\address{College of Mathematics, Xiamen University, Xiamen 361000, P.R. China}

\begin{abstract}
This paper studies the large time existence for the motion of closed
hypersurfaces in a radially symmetric potential. In physical, this surface can
be considered as an electrically charged membrane with a constant charge
per area in a radially symmetric potential. The evolution of such surface has
been investigated by Schn\"{u}rer and Smoczyk (Evolution of hypersurfaces in
central force fields, J. Reine Angew. Math. 550 (2002), 77-95). To study
its motion, we introduce a quasi-linear degenerate hyperbolic equation which
describes the motion of the surfaces extrinsically. Our main results show that
the large time existence of such Cauchy problem and the stability with respect
to small initial data. When the radially symmetric potential function $v\equiv1$,
the local existence and stability results have been obtained by Notz (Closed
Hypersurfaces driven by mean curvature and inner pressure, Comm. Pure
Appl. Math. 66(5) (2013), 790-819). The proof is based on a new Nash-Moser
iteration scheme.

\end{abstract}


\end{frontmatter}


\section{Introduction and Main Results}

Let $\Sigma$ be an oriented smooth closed manifold of dimension $n$,  and $(\mathcal{M},\tilde{g})$ be the Euclidean space, $i.e.$ $\mathcal{M}=\textbf{R}^{n+1}$ and $\tilde{g}$ be the Euclidean metric.
Consider a smooth family of immersions $F:[0,T]\times\Sigma\longrightarrow\mathcal{M}$, we define an action integral of the form
\begin{eqnarray*}
\mathcal{S}(F)=\int_0^T\mathcal{K}(F)-\mathcal{V}(F)dt,
\end{eqnarray*}
where $\mathcal{K}$ is the kinetic energy and $\mathcal{V}=V+\mathcal{J}$ is the total inner energy, $V$ is a radially symmetric potential energy and $\mathcal{J}$ is the inner pressure.

We fix a reference measure $d\mu$ on $\Sigma$ with a smooth density function defining a mass distribution on $\Sigma$. Then the total kinetic energy of all the points of the surface is
\begin{eqnarray*}
\mathcal{K}(F)=\frac{1}{2}\int_{\Sigma}|\partial_tF|^2d\mu.
\end{eqnarray*}

We denote $d\mu_t$  as the induced surface measure of the induced metric $g=F(t)^*\tilde{g}$ at time $t$, $\varphi$ denotes a smooth, radially
symmetric function (reflecting the presence of a central force) depending on $s:=\frac{|F|^2}{2}$. The radially symmetric potential energy of the hypersurface
$F(\Sigma)$ is defined by
\begin{eqnarray*}
V(F)=\int_{\Sigma}v(s(F))d\mu_t,
\end{eqnarray*}
where $v$ is defined by
\begin{eqnarray}\label{E1-0}
&&v(s)=\exp(-\frac{n}{2}\int_1^s\frac{\eta(w)}{w}dw),\\
&&\varphi(s)=-\frac{\partial_wv(w)}{v(w)}=\frac{n}{2w}\eta(w)\nonumber,
\end{eqnarray}
where $\eta:\textbf{R}^+\longrightarrow\textbf{R}$ is a smooth function.

The inner pressure is defined as
\begin{eqnarray*}
\mathcal{J}(F)=-\rho\log(\frac{Vol(F)}{Vol_0}),
\end{eqnarray*}
where $\rho>0$ denotes a parameter which determines strength of the influence of the inner pressure compared to the surface tension, $Vol(F)$ denotes the enclosed volume of the surface $F(\Sigma)$. The initial enclosed volume $Vol_0$ as well as the constant $\rho$ are included for scaling reasons. This inner pressure is motivated by that of an ideal gas with constant temperature, i.e. proportional to $Vol^{-1}(F)$. One can see \cite{Notz1,Notz} for more details on the inner pressure.

Then the action integral is
\begin{eqnarray*}
\mathcal{S}(F)=\frac{1}{2}\int_0^T\int_{\Sigma}|\partial_tF|^2d\mu dt-\int_0^T\int_{\Sigma}v(s(F))d\mu_tdt+\rho\int_0^T\log(\frac{Vol(F)}{Vol_0})dt.
\end{eqnarray*}
The Euler-Lagrange equations of functional $\mathcal{S}(F)$ is
\begin{eqnarray}\label{E1-1}
\nabla_{\partial_t}\partial_tF=\frac{d\mu_t}{d\mu}v(F)\left(-H(F)+\varphi F-\varphi\nabla(\frac{|F|^2}{2})+\frac{\rho}{Vol(F)}\right)\nu,
\end{eqnarray}
where $H(F)$ denotes the mean curvature of $F(\Sigma)$, $\nu$ denotes the outer unit normal, $|F|$ denotes the absolute value function of $F$,
$\nabla_{\partial_t}$ denotes the covariant derivative along $F$, i.e.
\begin{eqnarray*}
\nabla_{\partial_t}\partial_tF=\partial_{tt}F^{\alpha}+\tilde{\Gamma}^{\alpha}_{\beta\gamma}(F)\partial_tF^{\beta}\partial_tF^{\gamma},
\end{eqnarray*}
with $\tilde{\Gamma}^{\alpha}_{\beta\gamma}$ being the Christoffel symbols of $\tilde{g}$. Here and in the
sequel, we use the Einstein summation convention.

Another setting of above problem is that closed hypersurfaces moves in Riemannian manifolds with density $v(s)$. To our knowledge, Gromov \cite{Gromov}
first studied the manifolds with densities as ``mm-spaces'', and mentioned the natural generalization of mean curvature in such spaces. Morgan and his collaborators \cite{M1,M2} considered the corresponding mean curvature $H_{\omega}=H-<\overline{\nabla}\omega,\nu>$, where the density is denoted by $e^{\omega}$. When the density in the Euclidean space $\textbf{R}^{n+1}$ is the Gaussian density $(\frac{\gamma}{2\pi})^{\frac{n}{2}}e^{-\frac{\gamma|x|^2}{2}}$, Borisenko and Miquel \cite{Bori} studied a flow of a hypersurface driven by its mean curvature associated to the Gaussian density. It is obviously that the Gaussian density is a special case of the density $v(s)=\exp(-\frac{n}{2}\int_1^s\frac{\eta(w)}{w}dw)$.

The study of mean curvature flow can date back to Brakke \cite{Bra}, who introduced the
motion of a submanifold moving by its mean curvature in arbitrary codimension
and constructed a generalized varifold solution for all time. Huisken \cite{Hui1,Hui2} showed
that the mean curvature flow has much abundant and complicated behaviour, e.g. singularity and asymptotic behaviours.
Schn\"{u}rer and Smoczyk \cite{Sch} studied the asymptotic behaviour of mean curvature flow in central force fields, where they chose the radially symmetric potential as (\ref{E1-0}). For complex geometry, Chen and Tian \cite{Chen1} used the mean curvature flow in a K\"{a}hler-Einstein surface to show the symplectic property being preserved as long as the smooth solution exists. Chen and Li \cite{Chen2} produced holomorphic curves from a given initial symplectic surface by the mean curvature flow method. LeFloch and Smoczyk \cite{Lef} established a hyperbolic mean curvature flow, which is a strickly hyperbolic equation when the tangential part vanishing. Then they obtained the local existence and singularity for such flow.
Meanwhile, He and Kong \cite{He} also introduced a hyperbolic mean curvature flow and established the local existence and nonlinear stability for this kind of flow. After that, a hyperbolic mean curvature flow similar to the well-known Hamilton's Ricci flow was introduced by Kong and his collaborators \cite{Kong1,Kong2}, and they obtained the short time existence and lifespan result. Recently, A new kind of more physical hyperbolic mean curvature flow was established by Notz \cite{Notz}, its motivation is closed hypersurfaces moving driven by mean curvature and inner pressure. Meanwhile, the local existence of smooth solutions and the stability with respect to initial data was obtained.

We notice that the $d\mu$-term of equation (\ref{E1-1}) prevents reparametrization of (\ref{E1-1}) (see \cite{Deturck} for the method of reparametrization) to remove the degeneracy of such a quasilinear equation. Hence a suitable approximation method should be explored.
One of main results in this paper is the large time existence of equation (\ref{E1-1}).
\begin{theorem}
Let $F_0:\Sigma\longrightarrow\mathcal{M}=\textbf{R}^{n+1}$ be any smooth immersion with $Vol(F_0)=Vol_0>0$, and initial velocity $F_1\in\Gamma(F_0^*\textbf{T}M)$. Assume that
the radially symmetric potential function $v$ has the form (\ref{E1-0}), and
\begin{eqnarray}\label{E1-2}
0<|\frac{\eta(w)}{w}|\leq|w|^p,~~0\leq p<\infty.
\end{eqnarray}
Then there is $\epsilon>0$ and a smooth family of immerisions $F:[0,\frac{T}{\sqrt{\varepsilon}}]\longrightarrow\mathcal{M}=\textbf{R}^{n+1}$ solving the Cauchy problem (\ref{E1-1}) with small initial data $F(0,\cdot)=\varepsilon F_0$ and $\partial_tF(0,\cdot)=\varepsilon F_1$. Here $T$ is a positive constant and $\Gamma(F_0^*\textbf{T}M)$ denotes the space of smooth sections in a vector bundle $F_0^*\textbf{T}M$, and $\varepsilon$ is a positive sufficient small parameter.
\end{theorem}
We remark that we construct a small amplitude smooth solution, which can exist on time interval $[0,\frac{T}{\sqrt{\varepsilon}}]$ for a positive constant $T$. Here $\varepsilon$ is a sufficient small positive parameter measures the nonlinear effects.

Next result gives the stability and uniqueness of solutions for equation (\ref{E1-1}) with respect to the initial data in the Euclidean space $\textbf{R}^{n+1}$.
\begin{theorem}
Let $\bar{s}>\max\{2,\frac{n}{2}\}$. Assume that equation (\ref{E1-1}) with two different small initial data $(\bar{F}_0,\bar{F}_1)$ and $(\tilde{F}_0,\tilde{F}_1)$ has two different solutions $\bar{F}$ and $\tilde{F}$ in $\textbf{B}_{R,T}^s:=\{F\in\textbf{C}_T^s:~|||F|||_{s,T}\leq R<1\}$. If
\begin{eqnarray*}
|||\bar{F}_0-\tilde{F}_0|||_{\bar{s},T}+|||\bar{F}_1-\tilde{F}_1|||_{\bar{s},T}\leq \mathcal{O}(\varepsilon),
\end{eqnarray*}
then we have
\begin{eqnarray*}
|||\bar{F}-\tilde{F}|||_{\bar{s},T}\leq \mathcal{O}(\varepsilon).
\end{eqnarray*}
In particularly, when $\bar{F}_0=\tilde{F}_0$ and $\bar{F}_1=\tilde{F}_1$, the solution of equation (\ref{E1-1}) on the time interval $[0,\frac{T}{\sqrt{\varepsilon}}]$ is unique, i.e. $\bar{F}=\tilde{F}$. Here $\mathcal{O}(\varepsilon)$ denotes an infinitesimal and one can see the definition of the norm $|||\cdot|||_{s,T}$ in the section 3.
 \end{theorem}

The organization of this paper is as follows. In Section 2, we derive some conservation laws on the solution of equation (\ref{E1-1}).
In section 3, we establish the
existence of weakly linear hyperbolic system which will arise in the linearization
of equation (\ref{E1-1}). Section 5 is devoted to solving equation (\ref{E1-1}) by means of a new Nash-Moser iteration
scheme (For general Nash-Moser implict function theorem, one can see \cite{H,Moser,Nash}. In the section 5, we give the proof of the stability problem with respect to the initial data and the uniqueness of solution for equation (\ref{E1-1}). In the last section, we discuss the extension of our main results to more general case, i.e. $(\mathcal{M},g)$ be an oriented smooth complete $n+1$ dimensional Riemannian manifold.\\

\textbf{Acknowledgements}\\

This work is supported by the Fundamental Research Funds for the Central Universities of Xiamen University, No. 20720150013. This work was done in Beijing International Cen- ter for Mathematical Research, Peking University. The author expresses his sincere thanks to Prof. G. Tian for introducing him to the subject of hyperbolic mean curvature flows.

\section{Conservation Laws}\setcounter{equation}{0}

In this section, we derive three conservation laws, which are satisfied by the solution of (\ref{E1-1}). More precisely, they are the total energy conservation, the momentum conservation and the interior momentum conservation.

Let $X$ be a killing vector field on $\mathcal{M}=\textbf{R}^{n+1}$.
We define the total energy by
\begin{eqnarray*}
E(F(t,\cdot))=\frac{1}{2}\int_{\Sigma}|\partial_tF|^2d\mu-\int_{\Sigma}v(s(F))d\mu_t+\rho\log(\frac{Vol(F)}{Vol_0}),
\end{eqnarray*}
the momentum with respect to $X$ of a solution $F$ of (\ref{E1-1}) by
\begin{eqnarray*}
M_{X}(F(t,\cdot))=\int_{\Sigma}<\partial_tF,X(F)>d\mu,
\end{eqnarray*}
the interior momentum with respect to a vector field $Y$ ($div_{d\mu}Y=0$) by
\begin{eqnarray*}
Q_Y(F(t,\cdot))=\int_{\Sigma}<\partial_tF,F^*Y>d\mu.
\end{eqnarray*}

For the total energy conservation, we have:
\begin{lemma}
Let $F:[0,T)\times\Sigma\longrightarrow\mathcal{M}$ be a solution of (\ref{E1-1}). Then we have

(1). $E(F(t,\cdot))=E(F(0,\cdot))$ for all $t\in[0,T)$.

(2). Let $\partial_tF^T$ be the tangential part of $\partial_tF$. Then we have
\begin{eqnarray*}
\partial_t(\frac{1}{2}|\partial_tF|^2+v\frac{d\mu_t}{d\mu})=<\partial_tF,\nabla v>\frac{d\mu_t}{d\mu}+\frac{\rho}{Vol(F)}<\partial_tF,\nu>\frac{d\mu_t}{d\mu}+v(div\partial_tF^T)\frac{d\mu_t}{d\mu}.
\end{eqnarray*}
\end{lemma}
\begin{proof}
The proof of (1) is simple, so we omit it. Now we derive the equality in (2). Using (\ref{E1-1}), we derive
\begin{eqnarray*}
&&\partial_t(\frac{1}{2}|\partial_tF|^2+v\frac{d\mu_t}{d\mu})\nonumber\\
&=&<\partial_tF,\overline{\nabla}_{\partial_t}\partial_tF>+v(div\partial_tF^T)\frac{d\mu_t}{d\mu}
+v<\partial_tF, H\nu>\frac{d\mu_t}{d\mu}+\partial_tv\frac{d\mu_t}{d\mu}\nonumber\\
&=&-v<\partial_tF, H\nu>\frac{d\mu_t}{d\mu}+v<\partial_tF,\varphi F>\frac{d\mu_t}{d\mu}+<\partial_tF,\nabla v>\frac{d\mu_t}{d\mu}\nonumber\\
&&+\frac{\rho}{Vol(F)}<\partial_tF,\nu>\frac{d\mu_t}{d\mu}+v(div\partial_tF^T)\frac{d\mu_t}{d\mu}
+v<\partial_tF, H\nu>\frac{d\mu_t}{d\mu}+\partial_tv\frac{d\mu_t}{d\mu}\nonumber\\
&=&<\partial_tF,\nabla v>\frac{d\mu_t}{d\mu}+\frac{\rho}{Vol(F)}<\partial_tF,\nu>\frac{d\mu_t}{d\mu}+v(div\partial_tF^T)\frac{d\mu_t}{d\mu}.
\end{eqnarray*}
where we use the form of $v$ in the last inequality.
\end{proof}

\begin{lemma}
Assume that there is a positive constant $c$ such that
\begin{eqnarray*}
\int_{\Sigma}vd\mu_t\geq cVol(F^{\frac{n}{n+1}}).
\end{eqnarray*}
Then there exists a positive constant $K$ depending on $Vol_0$, $c$, $\rho$ and $E_0$ such that
\begin{eqnarray*}
Vol_0e^{-\frac{E_0}{\rho}}\leq Vol(F)\leq K,
\end{eqnarray*}
\begin{eqnarray*}
\int_{\Sigma}\frac{1}{2}|\partial_tF|^2d\mu+\int_{\Sigma}vd\mu_t\leq E_0+\rho\log(\frac{K}{Vol_0}).
\end{eqnarray*}
\end{lemma}
\begin{proof}
This proof is simple, and can also be found in \cite{Notz1}, so we omit it.
\end{proof}
For the momentum conservation, we have:
\begin{lemma}
Let $F:[0,T)\times\Sigma\longrightarrow\mathcal{M}$ be a solution of (\ref{E1-1}). Then we have

(1). $M_{X}(F(t,\cdot))=M_{X}(F(0,\cdot))$ for all $t\in[0,T)$.

(2). Let $\partial_tF^T$ be the tangential part of $\partial_tF$. Then we have
\begin{eqnarray*}
\partial_t<\partial_tF,X>=vdiv X^T\frac{d\mu_t}{d\mu}+<\nabla v\frac{d\mu_t}{d\mu},X>+\frac{\rho}{Vol(u)}<\nu,X>\frac{d\mu_t}{d\mu}.
\end{eqnarray*}
\end{lemma}
\begin{proof}
The momentum conservation (1) can be obtained by
integrating (2) with respect to $d\mu_0$ and $dt$ and using that
\begin{equation*}
\int_{\Sigma}<\nu,X>d\mu_t=\frac{\partial}{\partial_q}|_{q=0}Vol(F_q)=0.
\end{equation*}
So we only prove (2).

We denote $\phi_q$ as the local flow of $X$, which is by definition an isometry, and set $F_q=\phi_q\circ F$. Then we have
\begin{eqnarray*}
\partial_t<\partial_tF,X>&=&<\overline{\nabla}_{\partial_t}\partial_tF,X>+<\partial_tF,\overline{\nabla}_{\partial_t}X>\nonumber\\
&=&-v<H\nu,X>\frac{d\mu_t}{d\mu}+<v\varphi F\frac{d\mu_t}{d\mu},X>+<\nabla v\frac{d\mu_t}{d\mu},X>\nonumber\\
&&+\frac{\rho}{Vol(u)}<\nu,X>\frac{d\mu_t}{d\mu}\nonumber\\
&=&vdiv X^T\frac{d\mu_t}{d\mu}-\frac{\partial}{\partial_q}|_{q=0}\log(d\mu_t(F_q))\frac{d\mu_t}{d\mu}+\frac{\partial}{\partial_q}|_{q=0}v(s(F_q))\frac{d\mu_t}{d\mu}\nonumber\\
&&+<\nabla v\frac{d\mu_t}{d\mu},X>+\frac{\rho}{Vol(u)}<\nu,X>\frac{d\mu_t}{d\mu}\nonumber\\
&=&vdiv X^T\frac{d\mu_t}{d\mu}+<\nabla v\frac{d\mu_t}{d\mu},X>+\frac{\rho}{Vol(u)}<\nu,X>\frac{d\mu_t}{d\mu}.
\end{eqnarray*}
Furthermore, since $X$ is killing, we get
\begin{eqnarray*}
&&\frac{\partial}{\partial_q}|_{q=0}\log(d\mu_t(F_q))=0,\\
&&\frac{\partial}{\partial_q}|_{q=0}v(s(F_q))=0.
\end{eqnarray*}
\end{proof}

\begin{lemma}
Let $F:[0,T)\times\Sigma\longrightarrow\mathcal{M}$ be a solution of (\ref{E1-1}) and $\partial_tF^T$ be the tangential part of $\partial_tF$. Then we have

(1). $Q_{Y}(F(t,\cdot))=Q_{Y}(F(0,\cdot))$ for all $t\in[0,T)$.

(2). $\partial_t<\partial_tF,F_*Y>=\frac{1}{2}div_{d\mu}(|\partial_tF|^2Y)$.
\end{lemma}
\begin{proof}
This proof is similar with Proposition 2.4 in \cite{Notz}, so we omit it.
\end{proof}

\section{Analysis of linearized weakly hyperbolic systems}\setcounter{equation}{0}

This section is to discuss the linearized system of (\ref{E1-1}), which is a weakly hyperbolic linear systems. More precisely, it is obtained by decomposing with respect to time dependent subbundles into a system of coupled linear wave equations and linear ODEs. Here we prove the local existence of smooth solution of an abstract system for convenience.

Firstly, we give some definitions of the norm and some inequalities. The following setting is inspired by the work of \cite{Notz}. Let $\Psi:\Omega\longrightarrow\Sigma$ be a $k$-dimensional Riemannian vector bundle over $\Sigma$, and $\textbf{F}$ be the Fr\'{e}chet space $\textbf{C}^{\infty}([0,T]\times\Sigma,\Omega)$ of smooth time dependent sections of $\Omega$. Assume that we have an atlas of coordinate charts $(x_{\alpha},U_{\alpha})$ of $\Sigma$ such that $\alpha=1,\cdots,J$, $x_{\alpha}(U_{\alpha})=B_1(0)$, and the sets $x^{-1}_{\alpha}(B_1(0))$ cover $\Sigma$, where $B_1(0)\subset\textbf{R}^n$ is a ball with radius 1. Meanwhile, for each such chart there are smooth time dependent local sections $\nu_{A}^{(\alpha)}$, $A=1,\cdots,d'$, and $\tau_d^{(\alpha)}$, $d=1,\cdots,d''$, of $\Omega$ $(d'+d''=k)$ defined on the domain of the chart that together form a basis of the fiber over each point in $U_{\alpha}$. For any other chart $(x_{\beta},U_{\beta})$ with $U_{\alpha}\cap U_{\beta}\neq{\O}$, we assume that the $\nu_{A}^{(\alpha)}(p)$ and $\nu_{A}^{(\beta)}(p)$, $p\in U_{\alpha}\cap U_{\beta}$, span the same space and are bases for this space. The spaces spanned by the $\nu_A^{(\alpha)}$ and the $\tau_d^{(\alpha)}$ are orthogonal. For the specific (fixed) coordinate chart, we omit the index $(\alpha)$ for convenience. Let $h\in\textbf{F}$. In each coordinate chart we take the decompose
\begin{eqnarray*}
h=h_{\bot}+h_{\top}=u^A\nu_A+r^d\tau_d.
\end{eqnarray*}
For any $s\in[1,\infty]$, $\textbf{H}^s(B_1(0),\textbf{R}^{d'})^{\times J}$ be the Sobolev space of the set of functions $(u_{(\alpha)}^A)$. The corresponding norm is
\begin{eqnarray*}
\|u\|_s=\sum_{\alpha=1}^J\sum_{A=1}^{d'}\|u_{(\alpha)}^A\|_{\textbf{H}^s(B_1(0))}.
\end{eqnarray*}
For $\nu$, we use the norm
\begin{eqnarray*}
\|\nu\|_s=\sum_{\alpha=1}^J\sum_{A=1}^{d'}\|\nu_{\alpha}^A\circ x_{\alpha}^{-1}\|_{\textbf{H}^s(B_1(0))}.
\end{eqnarray*}
$\textbf{C}^{l}([0,T];\textbf{H}^s(B_1(0),\textbf{R}^{d'})^{\times J})$ denotes the
function spaces with the norm
\begin{eqnarray*}
\|u\|_{\textbf{C}^{l}([0,T];\textbf{H}^s(B_1(0),\textbf{R}^{d'})^{\times J})}=\sup_{[0,T]}\sum_{i=0}^l\sum_{\alpha=1}^J\sum_{A=1}^{d'}\|\partial_{t}^iu_{(\alpha)}^A\|_{\textbf{H}^s(B_1(0))}.
\end{eqnarray*}
$\textbf{C}_T^{s}:=\cap_{l=0}^i\textbf{C}^l([0,T];\textbf{H}^{s-i}(B_1(0),\textbf{R}^{d'})^{\times J})$ denotes function space with the spacetime norm
\begin{eqnarray*}
|||u|||_{s,T}=\sup_{[0,T]}\sum_{i=0}^2\sum_{\alpha=1}^J\sum_{A=1}^{d'}\|\partial_{t}^iu_{(\alpha)}^A\|_{\textbf{H}^{s-i}(B_1(0))}.
\end{eqnarray*}
We similarly define $\|\cdot\|_{\textbf{C}^s}$ as the spatial $\textbf{C}^s$-norm and the $\textbf{L}^2$-Sobolev norm of order $s$ in space and time.
It is obviously that
\begin{eqnarray*}
\|u\|_{\textbf{H}^s(B_1(0))}\leq\|u\|_{\textbf{C}^{0}([0,T];\textbf{H}^s(B_1(0)))}\leq|||u|||_{s,T}.
\end{eqnarray*}

We make use of the following inequalities, which can be proven using classical methods of calculus; see, for example, the book of Chapter 13 in \cite{Taylor}.
Note that while versions of these estimates hold in all dimensions, as presented here the estimates
are dependent on the dimension of $\Sigma$ being $n$. Generic constants are denoted by $c_0,c_1,\ldots$, their values may vary in the same
formula or in the same line.

Sobolev inequality: For $s>\frac{n}{2}$ and $u\in\textbf{H}^s(B_1(0))$ we have
\begin{eqnarray*}
\|u\|_{\textbf{L}^{\infty}(B_1(0))}\leq\|u\|_{\textbf{H}^{s}(B_1(0))}.
\end{eqnarray*}

Product inequality: For $u,v\in\textbf{L}^{\infty}(B_1(0))\cap\textbf{H}^{s}(B_1(0))$ we have
\begin{eqnarray*}
\|uv\|_{\textbf{H}^s(B_1(0))}\leq\|u\|_{\textbf{H}^{s}(B_1(0))}\|v\|_{\textbf{L}^{\infty}(B_1(0))}+\|u\|_{\textbf{L}^{\infty}(B_1(0))}\|v\|_{\textbf{H}^{s}(B_1(0))}.
\end{eqnarray*}
In particular, if $s>\frac{n}{2}$, we have
\begin{eqnarray*}
\|uv\|_{\textbf{H}^s(B_1(0))}\leq\|u\|_{\textbf{H}^{s}(B_1(0))}\|v\|_{\textbf{H}^{s}(B_1(0))}.
\end{eqnarray*}

We denote the spatially weighted Lebesgue spaces by $\textbf{L}^2_s(\textbf{R}^n)$, which is equipped with the norm
\begin{eqnarray*}
\|u\|_{\textbf{L}^2_s}=\int_{\textbf{R}^n}(1+|x|^2)^{\frac{l}{2}}|u|^2.
\end{eqnarray*}
Then the Fourier transform is an isomorphism between $\textbf{L}^2_s(\textbf{R}^n)$ and $\textbf{H}^s(\textbf{R}^n)$.

For $T>0$, $s>\max\{\frac{n}{2},2\}$ and $0<R<1$, we define
\begin{eqnarray*}
\textbf{B}_{R,T}^s:=\{u\in\textbf{C}_T^s=\cap_{l=0}^i\textbf{C}^l([0,T];\textbf{H}^{s-l}(B_1(0),\textbf{R}^{d'})^{\times J}):~|||u|||_{s,T}\leq R<1\}.
\end{eqnarray*}

We say that
\begin{eqnarray*}
h=u^A\nu_A+r^d\tau_d
\end{eqnarray*}
is a solution of a weakly hyperbolic linear system if in each local coordinate chart $(x_{\alpha},U_{\alpha})$ we have
\begin{eqnarray}\label{E2-1}
&&\partial_{tt}u^A-v(F)L^Au^A-v(F)N^Ar-v(F)Q^Au^A=h^A,\\
\label{E2-2}
&&\partial_{tt}r^d-v(F)M^dr-v(F)P^du=h^d,
\end{eqnarray}
with the initial data
\begin{eqnarray}\label{E2-3}
&&u^A(0)=u^A_0,~~\partial_tu^A(0)=u^A_1,\\
\label{E2-4}
&&r^d(0)=r^A_0,~~\partial_tr^d(0)=r^A_1,
\end{eqnarray}
where $v(F)=\exp(-\frac{n}{2}\int_1^{\frac{|F|^2}{2}}\frac{\eta(w)}{w}dw)$ and
\begin{eqnarray}
\label{E2-4-1}
\quad L^Au^A&=&a^{Aij}(t,x)\partial_i\partial_ju^A+a^{Ai}(t,x,\nu_A,\tau_d)\partial_iu^A+a^A(t,x,\nu_A,\tau_d)u^A,~~~~~\\
\label{E2-4-2}
N^Ar&=&n_{j}^{Ai}(t,x,\nu_A,\tau_d)\partial_ir^j+n_i^A(t,x,\nu_A,\tau_d)r^i+n_d^{A0}(t,x,\nu_A,\tau_d)\partial_tr^d\nonumber\\
&&+n^{A1}(t,x,\nu_A,\tau_d)\sum_{\beta=1}^J\int_{\Sigma}b_{(\beta)j}^Ar^j_{(\beta)}d\mu,\\
\label{E2-4-3}
Q^Au&=&q_{B}^{Ai}(t,x,\nu_A,\tau_d)\partial_iu^B+q_B^A(t,x,\nu_A,\tau_d)u^B+q_{B}^{A0}(t,x,\nu_A,\tau_d)\partial_tu^B\nonumber\\
&&+q^{A1}(t,x,\nu_A,\tau_d)\sum_{\beta=1}^J\int_{\Sigma}c_{(\beta)B}^Au^B_{(\beta)}d\mu,\\
\label{E2-4-4}
M^dr&=&m_i^d(t,x,\nu_A,\tau_d)r^i+m_i^{d0}(t,x,\nu_A,\tau_d)\partial_tr^i,\\
\label{E2-4-5}
P^du&=&P_B^{dj}(t,x,\nu_A,\tau_d)\partial_ju^B+P_B^{d0}(t,x,\nu_A,\tau_d)\partial_tu^B.
\end{eqnarray}
We assume all the coefficients in (\ref{E2-1})-(\ref{E2-2}) and $u^A$ and $r^d$ to be smooth functions on $x_{\alpha}(U_{\alpha})$, and $supp~ b_{(\beta)j}^A,~~supp~ c^A_{(\beta)B}\subset x^{-1}_{\beta}(B_1(0))$, $a^{Aij}(t,x)=a^{Aji}(t,x)$ and
$a^{Aij}(t,x)$ satisfies
\begin{eqnarray}\label{E2-5}
\rho_0\delta^{ij}|\xi|^2\leq a^{Aij}\xi_i\xi_j\leq\rho_1\delta^{ij}|\xi|^2,~~\forall \xi\in \textbf{R}^n,
\end{eqnarray}
for some positive constants $\rho_0\leq\rho_1$. On the other hand, we need that the operators are coordinate invariant under coordinate transformation on $\Sigma$ and under a change of basis between different $(\nu_A^{(\alpha)},\tau_d^{(\alpha)})$ and $(\nu_A^{(\beta)},\tau_d^{(\beta)})$. This means that the external force $W=h^A\nu_1+h^d\tau_d$ is an element of $\Omega$.

For the coefficient of the operators $L^A$, $N^A$, $Q^A$, $M^d$ and $P^d$, we define the local norm
\begin{eqnarray*}
[L]_s&=&\sum_{\alpha=1}^J[L]_{s,\alpha}=\sum_{\alpha=1}^J(\sum_{A,i,j}\|a^{Aij}\|_{\textbf{H}^s(\textbf{B}_1(0))}+\sum_{A,i}\|a^{Ai}\|_{\textbf{H}^s(\textbf{B}_1(0))}\\
&&+\sum_{A}\|a^{A}\|_{\textbf{H}^s(\textbf{B}_1(0))}),
\end{eqnarray*}
\begin{eqnarray*}
[N]_s&=&\sum_{\alpha=1}^J[N]_{s,\alpha}=\sum_{\alpha=1}^J(\sum_{A,i,j}\|n_{j}^{Ai}\|_{\textbf{H}^s(\textbf{B}_1(0))}+\sum_{A,i}\|n_i^A\|_{\textbf{H}^s(\textbf{B}_1(0))}\\
&&+\sum_{A,d}\|n_d^{A0}\|_{\textbf{H}^s(\textbf{B}_1(0))}+\sum_{A}\|n^{A1}\|_{\textbf{H}^s(\textbf{B}_1(0))}+\sum_{A,j}\|b_{(\alpha) j}^A\|_{\textbf{H}^s(\textbf{B}_1(0))}),
\end{eqnarray*}
\begin{eqnarray*}
[Q]_s&=&\sum_{\alpha=1}^J[Q]_{s,\alpha}=\sum_{\alpha=1}^J(\sum_{A,B,i}\|q_{B}^{Ai}\|_{\textbf{H}^s(\textbf{B}_1(0))}+\sum_{A,B}\|q_B^A\|_{\textbf{H}^s(\textbf{B}_1(0))},\\
&&+\sum_{A,B}\|q_{B}^{A0}\|_{\textbf{H}^s(\textbf{B}_1(0))}+\sum_{A}\|q^{A1}\|_{\textbf{H}^s(\textbf{B}_1(0))}+\sum_{A,j}\|c_{(\alpha) j}^A\|_{\textbf{H}^s(\textbf{B}_1(0))}),
\end{eqnarray*}
\begin{eqnarray*}
[M]_s=\sum_{\alpha=1}^J[M]_{s,\alpha}=\sum_{\alpha=1}^J(\sum_{d,i}\|m_i^d\|_{\textbf{H}^s(\textbf{B}_1(0))}+\sum_{d,i}\|m_i^{d0}\|_{\textbf{H}^s(\textbf{B}_1(0))}),
\end{eqnarray*}
\begin{eqnarray*}
[P]_s=\sum_{\alpha=1}^J[P]_{s,\alpha}=\sum_{\alpha=1}^J(\sum_{B,d,j}\|P_B^{dj}\|_{\textbf{H}^s(\textbf{B}_1(0))}+\sum_{B,d}\|P_B^{d0}\|_{\textbf{H}^s(\textbf{B}_1(0))}).
\end{eqnarray*}
Assume that
\begin{eqnarray}
\label{E2-5-1}
&&[L]_s+[\partial_tL]_s+[N]_{s}+[Q]_s+[M]_s+[P]_s\leq c_0,\\
\label{E2-5-2}
&&\det(<\nu_A,\nu_B>)>c_1,~~\det(<\tau_d,\tau_l>)>c_1.
\end{eqnarray}
We give some estimates on $\nu=\frac{\partial_1F\times\ldots\times\partial_nF}{|\partial_1F\times\ldots\times\partial_nF|}$ and $\tau_d=\partial_dF$. The proof of the following results is similar with the results in page 26, 49 of \cite{Notz1}, so we omit it.
\begin{lemma}
Let $F\in\textbf{B}_{R,T}^s$. Assume that  $\|\nu\|_{\textbf{C}^0}+\|\tau\|_{\textbf{C}^0}\leq R$ for some $R>0$ and (\ref{E2-5-2}) holds.
Then there exists a constant $c_1$ such that
\begin{eqnarray*}
&&\|\nu^{AB}\|_{\textbf{H}^s(\textbf{B}_1(0))}\leq c_1(1+\|\nu\|_{\textbf{H}^s(\textbf{B}_1(0))}),\\
&&\|\tau^{dl}\|_{\textbf{H}^s(\textbf{B}_1(0))}\leq c_1(1+\|\tau\|_{\textbf{H}^s(\textbf{B}_1(0))}),\\
&&|||\nu^{AB}|||_{\textbf{H}^s([0,T]\times\textbf{B}_1(0))}\leq c_1(1+\|\nu\|_{\textbf{H}^s([0,T]\times\textbf{B}_1(0))}),\\
&&|||\tau^{dl}|||_{\textbf{H}^s([0,T]\times\textbf{B}_1(0))}\leq c_1(1+\|\tau\|_{\textbf{H}^s([0,T]\times\textbf{B}_1(0))}),\\
&&\|\nu\|_{s}\leq c_1(1+\|F\|_{s+1}),~~~|||\nu|||_{s}\leq c_1(1+|||F|||_{s+1}),\\
&&\|\partial_t^i\nu\|_s\leq c_1(1+\sum_{l=0}^i\|\partial^l_tF\|_{s+1}),~~for~i>1,\\
&&\|\nu^{AB}\|_{\textbf{C}^0(\textbf{B}_1(0))}\leq c_1,~~\|\tau^{dl}\|_{\textbf{C}^0(\textbf{B}_1(0))}\leq c_1.
\end{eqnarray*}
\end{lemma}
From above result, we can see
\begin{eqnarray*}
\|\nu\|_{s+1}+\|\partial_t\nu\|_{s+1}+\|\partial_t^2\nu\|_{s}^2+\|\tau\|_{s+1}+\|\partial_t\tau\|_{s+1}\leq c_2.
\end{eqnarray*}

The following result is the main result in this section, which states the existence of linear system (\ref{E2-1})-(\ref{E2-2}).
\begin{proposition}
Assume that (\ref{E2-5})-(\ref{E2-5-2}) holds. Let $h_0,h_1\in\textbf{C}^{\infty}(\Sigma,\Omega)$ and $F\in\textbf{B}_{R,T}^s$ be given.
Then system (\ref{E2-1})-(\ref{E2-2}) has a unique smooth solution $h$ on $[0,T]\times\Sigma$ with $h(0)=h_0$ and $\partial_th(0)=h_1$ for some given external force $W=h^A\nu_1+h^d\tau_d\in\textbf{C}^{\infty}([0,T]\times\Sigma,\Omega)$.
\end{proposition}
Before giving the proof of above result,
we carry out some priori estimates on the solution $(u^A,r^d)$ of system (\ref{E2-1})-(\ref{E2-2}) in local coordinates. We remark that the following energy estimates does not depend on the compact property of the spatial domain. We also do not use the integral on spatial variable by part. The main idea of proof of Lemma 3.2-3.3 divides into two steps. The first step is to find a suitable differential inequality with respect to the time variables. The last step is to integrate the time variable by part.
Meanwhile, the integral on spatial variable is taken.
\begin{lemma}
Let $(u^A,r^d)$ be a smooth solution to system (\ref{E2-1})-(\ref{E2-2}) with initial data (\ref{E2-3})-(\ref{E2-4}).  $F\in\textbf{B}_{R,T}^s$ is given. Assume that (\ref{E2-5})-(\ref{E2-5-2}) holds. Then there holds
\begin{eqnarray}\label{E2-7}
&&\int_0^T\int_{B_1(0)}[\lambda-c_3v(F)-c_4n\frac{\eta(|F|^2)}{|F|^2}<\partial_iF+\partial_tF,F>v(F)]e^{-\lambda t}\nonumber\\
&&\times[(u^A_t)^2+a^{Aij}\partial_iu^A\partial_ju^A+(u^A)^2+(r_t^d)^2+(r^d)^2]dxdt\nonumber\\
&\leq&\int_{B_1(0)}e^{-\lambda t}((u^A_1)^2+a^{Aij}\partial_iu^A_0\partial_ju^A_0+(u^A_0)^2+(r_1^d)^2+(r^d_0)^2)dx\nonumber\\
&&+c_5\int_0^T\int_{B_1(0)}e^{-\lambda t}[(h^A)^2+(h^d)^2+(\partial_ir_t^d)^2]dxdt.
\end{eqnarray}
\end{lemma}
\begin{proof}
Taking the inner product of the linear system (\ref{E2-1}) with $2e^{-\lambda t}u^A_t$ and using (\ref{E2-4-1}), we have
\begin{eqnarray}\label{E2-8}
&&\partial_t[e^{-\lambda t}((u^A_t)^2+(F)a^{Aij}\partial_iu^A\partial_ju^A)]+\lambda e^{-\lambda t}((u^A_t)^2+v(F)a^{Aij}\partial_iu^A\partial_ju^A)\nonumber\\
&&-2e^{-\lambda t}\partial_i(v(F)a^{Aij}u^A_{t}\partial_ju^A)\nonumber\\
&=&v(F)e^{-\lambda t}(\partial_ta^{Aij})\partial_iu^A\partial_ju^A-2v(F)e^{-\lambda t}(\partial_ia^{Aij})u_t^A\partial_iu^A\nonumber\\
&&-2(\partial_tv(F)+\partial_iv(F))e^{-\lambda t}a^{Aij}\partial_iu^A\partial_ju^A\nonumber\\
&&+2v(F)e^{-\lambda t}a^{Ai}u_t^A\partial_iu^A+2v(F)e^{-\lambda t}a^Au_t^Au^A\nonumber\\
&&+2v(F)e^{-\lambda t}u^A_tN^Ar+2e^{-\lambda t}v(F)u_t^AQ^Au^A+2e^{-\lambda t}u_t^Ah^A.~~~~~~~~~~~
\end{eqnarray}
Taking the inner product of the linear system (\ref{E2-2}) with $2e^{-\lambda t}r^d_t$, we have
\begin{eqnarray}\label{E2-9}
&&\partial_t[e^{-\lambda t}(r_t^d)^2]+\lambda e^{-\lambda t}(r_t^d)^2\nonumber\\
&=&2e^{-\lambda t}v(F)r_t^dM^dr+2e^{-\lambda t}v(F)r_t^dP^du+2e^{-\lambda t}r_t^dh^d.
\end{eqnarray}
Since we have
\begin{eqnarray}\label{E2-10}
\quad\partial_t[e^{-\lambda t}(u^A)^2]+\lambda e^{-\lambda t}(u^A)^2=2e^{-\lambda t}u^Au^A_t,
\end{eqnarray}
\begin{eqnarray}
\label{E2-11}
\quad\partial_t[e^{-\lambda t}((r^d)^2+(\partial_ir^d)^2)]+\lambda e^{-\lambda t}[(r^d)^2+(\partial_ir^d)^2]=2e^{-\lambda t}(r^dr^d_t+\partial_ir^d\partial_ir_t^d),~~
\end{eqnarray}
and
\begin{eqnarray*}
\partial v(F)=-\frac{n}{2}\frac{\eta(|F|^2)}{|F|^2}<\partial F,F>v(F),
\end{eqnarray*}
summing up (\ref{E2-8})-(\ref{E2-11}), it holds
\begin{eqnarray}\label{E2-12}
&&\partial_t[e^{-\lambda t}((u^A_t)^2+v(F)a^{Aij}\partial_iu^A\partial_ju^A+(u^A)^2+(r_t^d)^2+(r^d)^2)]\nonumber\\
&&+\lambda e^{-\lambda t}[(u^A_t)^2+v(F)a^{Aij}\partial_iu^A\partial_ju^A+(u^A)^2+(r_t^d)^2+(r^d)^2]\nonumber\\
&&-2e^{-\lambda t}\partial_i(v(F)a^{Aij}u^A_{t}\partial_ju^A)\nonumber\\
&=&v(F)e^{-\lambda t}(\partial_ta^{Aij})\partial_iu^A\partial_ju^A-2v(F)e^{-\lambda t}(\partial_ia^{Aij})u_t^A\partial_iu^A\nonumber\\
&&+n\frac{\eta(|F|^2)}{|F|^2}<\partial_iF+\partial_tF,F>v(F)e^{-\lambda t}a^{Aij}\partial_iu^A\partial_ju^A\nonumber\\
&&+2v(F)e^{-\lambda t}a^{Ai}u_t^A\partial_iu^A+2v(F)e^{-\lambda t}a^Au_t^Au^A\nonumber\\
&&+2v(F)e^{-\lambda t}u^A_tN^Ar+2e^{-\lambda t}v(F)u_t^AQ^Au^A+2e^{-\lambda t}u_t^Ah^A\nonumber\\
&&+2e^{-\lambda t}v(F)r_t^dM^dr+2e^{-\lambda t}v(F)r_t^dP^du+2e^{-\lambda t}r_t^dh^d\nonumber\\
&&+2e^{-\lambda t}(u^Au^A_t+r^dr^d_t+\partial_ir^d\partial_ir_t^d).
\end{eqnarray}
Note that $v(F)\geq1$. Using Cauchy inequality, by (\ref{E2-4-2})-(\ref{E2-4-5}), (\ref{E2-5})-(\ref{E2-5-1}), (\ref{E2-12}) and Lemma 3.1, we derive
\begin{eqnarray}\label{E2-13}
&&\partial_t[e^{-\lambda t}((u^A_t)^2+v(F)a^{Aij}\partial_iu^A\partial_ju^A+(u^A)^2+(r_t^d)^2+(r^d)^2)]\nonumber\\
&&+\lambda e^{-\lambda t}[(u^A_t)^2+v(F)a^{Aij}\partial_iu^A\partial_ju^A+(u^A)^2+(r_t^d)^2+(r^d)^2]\nonumber\\
&&-2e^{-\lambda t}\partial_i(v(F)a^{Aij}u^A_{t}\partial_ju^A)\nonumber\\
&\leq&c_3v(F)e^{-\lambda t}[(u^A_t)^2+(\partial_iu^A)^2+(\partial_ju^A)^2+(u^A)^2+(r_t^d)^2+(r^d)^2]\nonumber\\
&&+c_4n\frac{\eta(|F|^2)}{|F|^2}<\partial_iF+\partial_tF,F>v(F)e^{-\lambda t}[(\partial_iu^A)^2+(\partial_ju^A)^2]\nonumber\\
&&+4e^{-\lambda t}(\partial_ir_t^d)^2+c_6e^{-\lambda t}[(h^A)^2+(h^d)^2].
\end{eqnarray}
Then inequality (\ref{E2-13}) leads to
\begin{eqnarray*}
&&\partial_t[e^{-\lambda t}((u^A_t)^2+a^{Aij}\partial_iu^A\partial_ju^A+(u^A)^2+(r_t^d)^2+(r^d)^2)]\nonumber\\
&&+[\lambda-c_3v(F)-c_4n\frac{\eta(|F|^2)}{|F|^2}<\partial_iF+\partial_tF,F>v(F)]e^{-\lambda t}\nonumber\\
&&\times[(u^A_t)^2+a^{Aij}\partial_iu^A\partial_ju^A+(u^A)^2+(r_t^d)^2+(r^d)^2]\nonumber\\
&&-2e^{-\lambda t}\partial_i(v(F)a^{Aij}u^A_{t}\partial_ju^A)\nonumber\\
&\leq&4e^{-\lambda t}(\partial_ir_t^d)^2+c_6e^{-\lambda t}[(h^A)^2+(h^d)^2].
\end{eqnarray*}
Thus integrating both side of above inequality on $[0,T]\times B_1(0)$, we obtain
\begin{eqnarray*}
&&\int_{B_1(0)}e^{-\lambda t}[(u^A_t(T))^2+a^{Aij}\partial_iu^A(T)\partial_ju^A(T)+(u^A(T))^2\nonumber\\
&&+(r_t^d(T))^2+(r^d(T))^2]dx\nonumber\\
&&+\int_0^T\int_{B_1(0)}[\lambda-c_3v(F)-c_4n\frac{\eta(|F|^2)}{|F|^2}<\partial_iF+\partial_tF,F>v(F)]e^{-\lambda t}\nonumber\\
&&\times[(u^A_t)^2+a^{Aij}\partial_iu^A\partial_ju^A+(u^A)^2+(r_t^d)^2+(r^d)^2]dxdt\nonumber\\
&\leq&\int_{B_1(0)}e^{-\lambda t}[(u^A_1)^2+a^{Aij}\partial_iu^A_0\partial_ju^A_0+(u^A_0)^2+(r_1^d)^2+(r^d_0)^2]dx\nonumber\\
&&+4\int_0^T\int_{B_1(0)}e^{-\lambda t}(\partial_ir_t^d)^2+c_6\int_0^T\int_{B_1(0)}e^{-\lambda t}[(h^A)^2+(h^d)^2]dxdt.
\end{eqnarray*}
This completes the proof.
\end{proof}

\begin{lemma}
Let $(u^A,r^d)$ be a smooth solution to system (\ref{E2-1})-(\ref{E2-2}) with initial data (\ref{E2-3})-(\ref{E2-4}). $F\in\textbf{B}_{R,T}^s$ is given. Assume that (\ref{E1-2}) and (\ref{E2-5})-(\ref{E2-5-2}) hold.
Then there exists a positive constant $\lambda>c_7$ such that
\begin{eqnarray}\label{E2-14}
&&(\lambda-c_7)\int_0^T\int_{B_1(0)}e^{-\lambda t}[(u^A_t)^2+(\partial_iu^A)^2+(u^A)^2+(r_t^d)^2+(r^d)^2]dxdt\nonumber\\
&\leq&c_8\int_{B_1(0)}e^{-\lambda t}[(u^A_0)^2+(\partial_iu^A_0)^2+(u^A_1)^2+(\partial_iu^A_1)^2\nonumber\\
&&+(r_0^d)^2+(\partial_ir_0^d)^2+(r^d_1)^2+(\partial_ir_1^d)^2]dx\nonumber\\
&&+c_9\int_0^T\int_{B_1(0)}e^{-\lambda t}(|h^A|^2+|\partial_ih^A|^2+|h^d|^2+|\partial_ih^d|^2)dxdt.
\end{eqnarray}
\end{lemma}
\begin{proof}
we introduce auxiliary functions $\tilde{u}^A$ and $\tilde{r}^d$ , which satisfy
\begin{eqnarray}\label{E2-15}
&&\partial_{tt}\tilde{u}^A-\partial_t\tilde{u}^A-\tilde{u}^A=h^A,\\
&&\tilde{u}^A(0,x)=u^A_0,~~\partial_t\tilde{u}^A(0,x)=u^A_1,\nonumber
\end{eqnarray}
and
\begin{eqnarray}
\label{E2-15-1}
&&\partial_{tt}\tilde{r}^d-\partial_t\tilde{r}^d-\tilde{r}^d=h^d,\\
&&\tilde{r}^d(0,x)=r^d_0,~~\partial_t\tilde{r}^d(0,x)=r^d_1.\nonumber
\end{eqnarray}

Let
\begin{eqnarray}\label{E2-16-1}
&&\hat{u}^A=u^A-\tilde{u}^A,\\
\label{E2-16-2}
&&\hat{r}^d=r^d-\tilde{r}^d,
\end{eqnarray}
then it follows from (\ref{E2-1})-(\ref{E2-2}) and (\ref{E2-15})-(\ref{E2-15-1}) that
\begin{eqnarray}\label{E2-16}
&&\partial_{tt}\hat{u}^A-v(F)L^A\hat{u}^A-v(F)N^A\hat{r}-v(F)Q^A\hat{u}^A=\hat{g}_1,\\
&&\hat{u}^A(0,x)=0,~~\partial_t\hat{u}^A(0,x)=0,\nonumber
\end{eqnarray}
\begin{eqnarray}
\label{E2-17}
&&\partial_{tt}\hat{r}^d-v(F)M^d\hat{r}-v(F)P^d\hat{u}=\hat{g}_2,\\
&&\hat{r}^d(0,x)=0,~~\partial_t\hat{r}^d(0,x)=0,\nonumber
\end{eqnarray}
where $L^A$, $N^A$, $Q^A$, $M^d$ and $P^d$ are defined in (\ref{E2-4-1})-(\ref{E2-4-5}), and
\begin{eqnarray}\label{E2-18}
&&\hat{g}_1=v(F)L^A\tilde{u}^A+v(F)N^A\tilde{r}+v(F)Q^A\tilde{u}^A-\partial_t\tilde{u}^A-\tilde{u}^A,\\
\label{E2-19}
&&\hat{g}_2=v(F)M^d\tilde{r}+v(F)P^d\tilde{u}-\partial_t\tilde{r}^d-\tilde{r}^d.
\end{eqnarray}
Using the similar deriving process with (\ref{E2-13}), from (\ref{E2-16})-(\ref{E2-17}) we get
\begin{eqnarray}\label{E2-20}
&&\partial_t[e^{-\lambda t}((\hat{u}^A_t)^2+v(F)a^{Aij}\partial_i\hat{u}^A\partial_j\hat{u}^A+(\hat{u}^A)^2+(\hat{r}_t^d)^2+(\hat{r}^d)^2)]\nonumber\\
&&+\lambda e^{-\lambda t}[(\hat{u}^A_t)^2+v(F)a^{Aij}\partial_i\hat{u}^A\partial_j\hat{u}^A+(\hat{u}^A)^2+(\hat{r}_t^d)^2+(\hat{r}^d)^2]\nonumber\\
&&-2e^{-\lambda t}\partial_i(v(F)a^{Aij}\hat{u}^A_{t}\partial_j\hat{u}^A)\nonumber\\
&=&v(F)e^{-\lambda t}(\partial_ta^{Aij})\partial_i\hat{u}^A\partial_j\hat{u}^A-2v(F)e^{-\lambda t}(\partial_ia^{Aij})\hat{u}_t^A\partial_i\hat{u}^A\nonumber\\
&&+n\frac{\eta(|F|^2)}{|F|^2}<\partial_iF+\partial_tF,F>v(F)e^{-\lambda t}a^{Aij}\partial_i\hat{u}^A\partial_j\hat{u}^A\nonumber\\
&&+2v(F)e^{-\lambda t}a^{Ai}\hat{u}_t^A\partial_i\hat{u}^A+2v(F)e^{-\lambda t}a^A\hat{u}_t^A\hat{u}^A\nonumber\\
&&+2v(F)e^{-\lambda t}\hat{u}^A_tN^A\hat{r}+2e^{-\lambda t}v(F)\hat{u}_t^AQ^A\hat{u}^A+2e^{-\lambda t}v(F)\hat{r}_t^dM^d\hat{r}\nonumber\\
&&+2e^{-\lambda t}v(F)\hat{r}_t^dP^d\hat{u}+2e^{-\lambda t}(\hat{u}^A\hat{u}^A_t+\hat{r}^d\hat{r}^d_t+\partial_i\hat{r}^d\partial_i\hat{r}_t^d)\nonumber\\
&&+2e^{-\lambda t}\hat{u}_t^A\hat{g}_1+2e^{-\lambda t}\hat{r}_t^d\hat{g}_2.
\end{eqnarray}
In order to avoid an extra loss of derivatives of integrating (\ref{E2-20}), we need to rewrite (\ref{E2-20}).
By (\ref{E2-18})-(\ref{E2-19}), the last two terms in (\ref{E2-20}) can be rewritten as
\begin{eqnarray}\label{E2-21}
&&e^{-\lambda t}(\hat{u}_t^A\hat{g}_1+\hat{r}_t^d\hat{g}_2)\nonumber\\
&=&\partial_i(v(F)e^{-\lambda t}a^{Aij}\partial_i\tilde{u}^A\partial_t\hat{u}^A)-\partial_t(v(F)e^{-\lambda t}a^{Aij}\partial_i\tilde{u}^A\partial_j\hat{u}^A)\nonumber\\
&&-v(F)e^{-\lambda t}\partial_ia^{Aij}\partial_i\hat{u}^A\partial_t\tilde{u}^A-\lambda v(F)e^{-\lambda t}a^{Aij}\partial_i\tilde{u}^A\partial_j\hat{u}^A\nonumber\\
&&+v(F)e^{-\lambda t}\partial_ta^{Aij}\partial_i\tilde{u}^A\partial_j\hat{u}^A+v(F)e^{-\lambda t}a^{Aij}\partial_i\hat{u}^A\partial_t\partial_j\tilde{u}^A\nonumber\\
&&-\frac{\eta(|F|^2)}{2|F|^2}<\partial_iF-\partial_tF,F>v(F)e^{-\lambda t}a^{Aij}\partial_i\tilde{u}^A\partial_j\hat{u}^A+v(F)e^{-\lambda t}a^{Ai}\partial_i\tilde{u}^A\hat{u}_t\nonumber\\
&&+v(F)e^{-\lambda t}a^{A}\tilde{u}^A\hat{u}_t+e^{-\lambda t}\hat{u}^A_t(v(F)N^A\tilde{r}+v(F)Q^A\tilde{u}^A-\partial_t\tilde{u}^A-\tilde{u}^A)\nonumber\\
&&+e^{-\lambda t}\hat{r}^d_t(v(F)M^d\tilde{r}+v(F)P^d\tilde{u}-\tilde{r}^d_t-\tilde{r}^d).
\end{eqnarray}
Note that $v(F)\geq1$. Inserting (\ref{E2-21}) into (\ref{E2-20}), integrating in $[0,T]\times B_1(0)$, then using Cauchy inequality we obtain
\begin{eqnarray}\label{E2-22}
&&\int_{B_1(0)}e^{-\lambda t}[(\hat{u}^A_t(T))^2+v(F)a^{Aij}\partial_i\hat{u}^A(T)\partial_j\hat{u}^A(T)+(\hat{u}^A(T))^2+(\hat{r}_t^d(T))^2\nonumber\\
&&+(\hat{r}^d(T))^2]dx+\int_0^T\int_{B_1(0)}[\lambda-|\frac{\eta(|F|^2)}{|F|^2}<\partial_iF+\partial_tF,F>|v(F)]\nonumber\\
&&\times e^{-\lambda t}[(\hat{u}^A_t)^2+v(F)a^{Aij}\partial_i\hat{u}^A\partial_j\hat{u}^A+(\hat{u}^A)^2+(\hat{r}_t^d)^2+(\hat{r}^d)^2]dxdt\nonumber\\
&\leq&\int_{B_1(0)}e^{-\lambda t}v(F)(\partial_i\tilde{u}^A(T))^2dx\nonumber\\
&&+c_{10}n\int_0^T\int_{B_1(0)}e^{-\lambda t}|\frac{\eta(|F|^2)}{|F|^2}<\partial_iF-\partial_tF,F>|v(F)(\partial_i\tilde{u}^A)^2dxdt\nonumber\\
&&+c_{11}\int_0^T\int_{B_1(0)}e^{-\lambda t}v(F)[(\partial_i\tilde{u}^A_t)^2+(\partial_i\tilde{u}^A)^2+(\tilde{u}^A_t)^2+(\tilde{u}^A)^2dxdt\nonumber\\
&&+(\tilde{r}^d_t)^2+(\tilde{r}^d)^2+(\partial_i\tilde{r}^d_t)^2]dxdt.
\end{eqnarray}
So it follows from (\ref{E2-16-1})-(\ref{E2-16-2}) and (\ref{E2-22}) that
\begin{eqnarray}\label{E2-23}
&&\int_0^T\int_{B_1(0)}[\lambda-|\frac{\eta(|F|^2)}{|F|^2}<\partial_iF+\partial_tF,F>|v(F)]\nonumber\\
&&\times e^{-\lambda t}[(u^A_t)^2+v(F)a^{Aij}\partial_iu^A\partial_ju^A+(u^A)^2+(r_t^d)^2+(r^d)^2]dxdt\nonumber\\
&\leq&\int_{B_1(0)}e^{-\lambda t}v(F)(\partial_i\tilde{u}^A(T))^2dx+\int_0^T\int_{B_1(0)}[\lambda-|\frac{\eta(|F|^2)}{|F|^2}<\partial_iF+\partial_tF,F>|v(F)]\nonumber\\
&&\times e^{-\lambda t}[(\tilde{u}^A_t)^2+v(F)a^{Aij}\partial_i\tilde{u}^A\partial_j\tilde{u}^A+(\tilde{u}^A)^2+(\tilde{r}_t^d)^2+(\tilde{r}^d)^2]dxdt\nonumber\\
&&+c_{10}n\int_0^T\int_{B_1(0)}e^{-\lambda t}|\frac{\eta(|F|^2)}{|F|^2}<\partial_iF-\partial_tF,F>|v(F)(\partial_i\tilde{u}^A)^2dxdt\nonumber\\
&&+c_{11}\int_0^T\int_{B_1(0)}e^{-\lambda t}v(F)[(\partial_i\tilde{u}^A_t)^2+(\partial_i\tilde{u}^A)^2+(\tilde{u}^A_t)^2+(\tilde{u}^A)^2\nonumber\\
&&\quad\quad +(\tilde{r}^d_t)^2+(\tilde{r}^d)^2+(\partial_i\tilde{r}^d_t)^2]dxdt.
\end{eqnarray}
Multiplying (\ref{E2-15}) and (\ref{E2-15-1}) both side by $2e^{-\lambda t}\tilde{u}^A$ and $2e^{-\lambda t}\tilde{r}^d$, respectively, then
integrating on $[0,T]\times B_1(0)$ and summing up two inequalities, we have
\begin{eqnarray}\label{E2-24}
&&\int_{B_1(0)}e^{-\lambda t}[(\tilde{u}^A(T))^2+(\partial_t\tilde{u}^A(T))^2+(\tilde{r}^d(T))^2+(\partial_t\tilde{r}^d(T))^2]dx\nonumber\\
&&+\lambda\int_0^T\int_{B_1(0)}e^{-\lambda t}[(\tilde{u}^A)^2+(\partial_t\tilde{u}^A)^2+(\tilde{r}^d)^2+(\partial_t\tilde{r}^d)^2]dxdt\nonumber\\
&\leq& c_{12}\int_{B_1(0)}e^{-\lambda t}[(u^A_0)^2+(u^A_1)^2+(r^d_0)^2+(r^d_1)^2]dx\nonumber\\
&&+c_{13}\int_0^T\int_{B_1(0)}e^{-\lambda t}(|h^A|^2+|h^d|^2)dxdt.~~~~~
\end{eqnarray}
Differentiating (\ref{E2-15}) with respect to $x$, by the similar process of getting (\ref{E2-24}), we derive
\begin{eqnarray}\label{E2-25}
&&\int_{B_1(0)}e^{-\lambda t}[(\partial_i\tilde{u}^A(T))^2+(\partial_t\partial_i\tilde{u}^A(T))^2+(\partial_i\tilde{r}^d(T))^2+(\partial_t\partial_i\tilde{r}^d(T))^2]dx\nonumber\\
&&+\lambda\int_0^T\int_{B_1(0)}e^{-\lambda t}[(\partial_i\tilde{u}^A)^2+(\partial_t\partial_i\tilde{u}^A)^2+(\partial_i\tilde{r}^d)^2+(\partial_t\partial_i\tilde{r}^d)^2]dxdt\nonumber\\
&\leq&c_{14}\int_{B_1(0)}e^{-\lambda t}[(u^A_0)^2+(u^A_1)^2+(\partial_iu_0^A)^2+(\partial_iu_1^A)^2]dx\nonumber\\
&&+c_{15}\int_{B_1(0)}e^{-\lambda t}[(r^d_0)^2+(r^d_1)^2+(\partial_ir_0^d)^2+(\partial_ir_1^d)^2]dx\nonumber\\
&&+c_{16}\int_0^T\int_{B_1(0)}e^{-\lambda t}(|h^A|^2+|\partial_ih^A|^2+|h^d|^2+|\partial_ih^d|^2)dxdt.
\end{eqnarray}
On the other hand, we notice that
\begin{eqnarray*}
F\in\textbf{B}_{R,T}^s,~~~~\|F\|_{\textbf{L}^{\infty}}\leq\|F\|_{\textbf{H}^s}~~~and~~~\|F\|_{C^0([0,T];\textbf{H}^s)}\leq|||F|||_{T,s}~~for~~s\geq1.
\end{eqnarray*}
So by (\ref{E1-2}), we have
\begin{eqnarray*}
|\frac{\eta(|F|^2)}{|F|^2}<\partial_iF+\partial_tF,F>|&\leq&|||F|||_{T,s}^{2p+1}(\|F\|_{\textbf{C}^0([0,T];\textbf{H}^{s+1})}+\|F\|_{\textbf{C}^1([0,T];\textbf{H}^{s})})\\
&\leq&2|||F|||_{s+1,T}^{2p+2}\leq R^{2p+2}<1,
\end{eqnarray*}
\begin{eqnarray}
\label{E2-26-0}
v(F)&\leq&e^{\frac{n}{2}\int_1^s|\sigma|^pd\sigma}\leq e^{\frac{n}{2^{(p+2)}(p+1)}||||F||||_{s,T}^{2(p+1)}}\leq e^{\frac{n}{2^{(p+2)}(p+1)}R^{2(p+1)}}.
\end{eqnarray}
Hence by (\ref{E2-23}), there exists a $\lambda$ such that
\begin{eqnarray*}
\lambda>e^{\frac{n}{2^{(p+2)}(p+1)}R^{2(p+1)}}>1,
\end{eqnarray*}
and
\begin{eqnarray}\label{E2-26}
&&(\lambda-e^{\frac{n}{2^{(p+2)}(p+1)}R^{2(p+1)}})\int_0^T\int_{B_1(0)}e^{-\lambda t}[(u^A_t)^2+e^{\frac{n}{2^{(p+2)}(p+1)}R^{2(p+1)}}a^{Aij}\partial_iu^A\partial_ju^A\nonumber\\
&&+(u^A)^2+(r_t^d)^2+(r^d)^2]dxdt\nonumber\\
&\leq&e^{\frac{n}{2^{(p+2)}(p+1)}R^{2(p+1)}}\int_{B_1(0)}e^{-\lambda t}(\partial_i\tilde{u}^A(T))^2dxdt\nonumber\\
&&+(\lambda-e^{\frac{n}{2^{(p+2)}(p+1)}R^{2(p+1)}})\int_0^T\int_{B_1(0)}e^{-\lambda t}[(\tilde{u}^A_t)^2+e^{\frac{n}{2^{(p+2)}(p+1)}R^{2(p+1)}}a^{Aij}\partial_i\tilde{u}^A\partial_j\tilde{u}^A\nonumber\\
&&+(\tilde{u}^A)^2+(\tilde{r}_t^d)^2+(\tilde{r}^d)^2]dxdt\nonumber\\
&&+c_{17}ne^{-\lambda t}e^{\frac{n}{2^{(p+2)}(p+1)}R^{2(p+1)}}\int_0^T\int_{B_1(0)}(\partial_i\tilde{u}^A)^2dxdt\nonumber\\
&&+c_{18}e^{\frac{n}{2^{(p+2)}(p+1)}R^{2(p+1)}}\int_0^T\int_{B_1(0)}e^{-\lambda t}[(\partial_i\tilde{u}^A_t)^2+(\partial_i\tilde{u}^A)^2+(\tilde{u}^A_t)^2+(\tilde{u}^A)^2\nonumber\\
&&+(\tilde{r}^d_t)^2+(\tilde{r}^d)^2+(\partial_i\tilde{r}^d_t)^2]dxdt.
\end{eqnarray}
On the other hand, inequalities (\ref{E2-24})-(\ref{E2-25}) give the control of four terms in the right hand side of (\ref{E2-26}). Thus substituting (\ref{E2-24})-(\ref{E2-25}) in (\ref{E2-26}), we obtain (\ref{E2-14}). This completes the proof.
\end{proof}
In what follows, we plan to obtain the estimate for $|||u^A|||_{s,T}$ and $|||r^d|||_{s,T}$ by considering the equation of time space derivatives of $u^A$ and $r^d$.
For convenience, we recall the norm $|||u|||_{s,T}=\sup_{[0,T]}\sum_{i=0}^2\sum_{\alpha=1}^J\sum_{A=1}^{d'}\|\partial_{t}^iu_{(\alpha)}^A\|_{\textbf{H}^{s-i}(B_1(0))}$.
For any multi-index $\gamma\in\textbf{Z}^2_+$ with $|\gamma|=s$, applying $\partial_i^{\gamma}$ to both sides of (\ref{E2-1}) and (\ref{E2-2}) to get
\begin{eqnarray}\label{E2-27}
&&\partial_{tt}\partial_i^{\gamma}u^A-v(F)L^A\partial_i^{\gamma}u^A-v(F)N^A\partial_i^{\gamma}r-v(F)Q^A\partial_i^{\gamma}u^A=h_{\gamma}^A,\\
\label{E2-28}
&&\partial_{tt}\partial_i^{\gamma}r^d-v(F)M^d\partial_i^{\gamma}r-v(F)P^d\partial_i^{\gamma}u=h^d_{\gamma},
\end{eqnarray}
where the nonlinear terms
\begin{eqnarray}\label{E2-27R0}
h^A_{\gamma}&=&\partial_i^{\gamma}h^A+\sum_{s_1+s_2=\gamma,~|s_1|\geq1}\partial_i^{s_1}(v(F)a^{A{ij}})\partial_{i}\partial_j\partial_i^{s_2}u^A\nonumber\\
&&+\sum_{s_1+s_2=\gamma,~|s_1|\geq1}\partial_i^{s_1}(v(F)a^{Ai})\partial_i\partial_i^{s_2}u^A\nonumber\\
&&+\sum_{s_1+s_2=\gamma,~|s_1|\geq1}\partial_i^{s_1}(v(F)a^{A})D^{s_2}u^A+\sum_{s_1+s_2=\gamma,~|s_1|\geq1}\partial_i^{s_1}(v(F)n^{Ai}_j)\partial_i\partial_i^{s_2}r^j\nonumber\\
&&+\sum_{s_1+s_2=\gamma,~|s_1|\geq1}\partial_i^{s_1}(v(F)n^{A}_i)D^{s_2}r^j+\sum_{s_1+s_2=\gamma,~|s_1|\geq1}\partial_i^{s_1}(v(F)n^{A0}_d)\partial_i^{s_2}r_t^d\nonumber\\
&&+\sum_{s_1+s_2=\gamma,~|s_1|\geq1}\partial_i^{s_1}(v(F)n^{A1})\sum_{\beta=1}^J\int_{\Sigma}\partial_i^{s_2}(b_{(\beta)}^Ar^j_{(\beta)})d\mu\nonumber\\
&&+\sum_{s_1+s_2=\gamma,~|s_1|\geq1}\partial_i^{s_1}(v(F)q^{Ai}_B)\partial_i\partial_i^{s_2}u^B+\sum_{s_1+s_2=\gamma,~|s_1|\geq1}\partial_i^{s_1}(v(F)q^{A}_B)\partial_i^{s_2}u^B\nonumber\\
&&+\sum_{s_1+s_2=\gamma,~|s_1|\geq1}\partial_i^{s_1}(v(F)q^{A0}_B)\partial_i^{s_2}u_t^B\nonumber\\
&&+\sum_{s_1+s_2=\gamma,~|s_1|\geq1}\partial_i^{s_1}(v(F)q^{A1})\sum_{\beta=1}^J\int_{\Sigma}\partial_i^{s_2}(c_{(\beta)B}^Au^B_{(\beta)})d\mu
\end{eqnarray}
and
\begin{eqnarray}\label{E2-27R1}
h^d_{\gamma}&=&\partial_i^{\gamma}h^d+\sum_{s_1+s_2=\gamma,~|s_1|\geq1}\partial_i^{s_1}(v(F)m^{d}_i)\partial_i^{s_2}r^i+\sum_{s_1+s_2=\gamma,~|s_1|\geq1}\partial_i^{s_1}(v(F)m_{i}^{d0})\partial_i^{s_2}\gamma_t^i\nonumber\\
&&+\sum_{s_1+s_2=\gamma,~|s_1|\geq1}\partial_i^{s_1}(v(F)P^{dj}_B))\partial_j\partial_i^{s_2}u^B\nonumber\\
&&+\sum_{s_1+s_2=\gamma,~|s_1|\geq1}\partial_i^{s_1}(v(F)P_{B}^{k0})\partial_i^{s_2}u_t^B.
\end{eqnarray}
For convenience, we denote all spacial derivatives of $u^A$ and $r^d$ of the order $s$ by a column vector of $m(s)$ components and $n(s)$ components
\begin{eqnarray*}
&&u_s^T=(\partial_1^su^A,\partial_1^{s-1}\partial_2u^A,\ldots,\partial_2^su^A),\\
&&r_s^T=(\partial_1^sr^d,\partial_1^{s-1}\partial_2r^d,\ldots,\partial_2^sr^d).
\end{eqnarray*}
Putting (\ref{E2-27})-(\ref{E2-28}) together corresponding to all $\gamma$ with $|\gamma|=s$, we have
\begin{eqnarray}\label{E2-29-1}
&&\partial_{tt}u_s-v(F)\tilde{L}^Au_s-v(F)\tilde{N}^Ar_s-v(F)\tilde{Q}^Au_s=\tilde{h}_{\gamma}^A,\\
\label{E2-30-1}
&&\partial_{tt}r_s-v(F)\tilde{M}^dr_s-v(F)\tilde{P}^du_s=\tilde{h}^d_{\gamma},
\end{eqnarray}
where
\begin{eqnarray*}
&&\tilde{L}^Au_s=\tilde{a}^{Aij}\partial_i\partial_ju_s+\tilde{a}^{Ai}\partial_iu_s+\tilde{a}^Au_s,\\
&&\tilde{N}^Ar_s=\tilde{n}_{j}^{Ai}\partial_ir_s+\tilde{n}_i^Ar_s+\tilde{n}_d^{A0}\partial_tr_s+\tilde{n}^{A1}\sum_{\beta=1}^J\int_{\Sigma}\tilde{b}_{(\beta)j}^Ar_{s(\beta)}d\mu,\\
&&\tilde{Q}^Au_s=\tilde{q}_{B}^{Ai}\partial_iu_s+\tilde{q}_B^Au_s+\tilde{q}_{B}^{A0}\partial_tu_s+\tilde{q}^{A1}\sum_{\beta=1}^J\int_{\Sigma}\tilde{c}_{(\beta)B}^Au_{s(\beta)}d\mu,\\
&&\tilde{M}^dr_s=\tilde{m}_i^dr_s+\tilde{m}_i^{d0}\partial_tr_s,\\
&&\tilde{P}^du_s=\tilde{P}_B^{dj}\partial_ju_s+\tilde{P}_B^{d0}\partial_tu_s.
\end{eqnarray*}
with all the coefficients in $\tilde{L}^A$, $\tilde{N}^A$, $\tilde{Q}^A$, $\tilde{M}^d$ and $\tilde{P}^d$
are $(m\times m)$-matrices and $\tilde{h}_{\gamma}^A$ and $\tilde{h}^d_{\gamma}$ are $m$-vectors given by
\begin{eqnarray*}
&&\tilde{h}_{\gamma}^A=(h^A_{(s,0,\ldots,0)},h^A_{(s-1,1,\ldots,0)},\ldots,h^A_{(0,0,\ldots,s)})^T,\\
&&\tilde{h}_{\gamma}^d=(h^d_{(s,0,\ldots,0)},h^d_{(s-1,1,\ldots,0)},\ldots,h^d_{(0,0,\ldots,s)})^T.
\end{eqnarray*}
\begin{lemma}
Let $(u^A,r^d)$ be a smooth solution to (\ref{E2-1})-(\ref{E2-2}) with initial data (\ref{E2-3})-(\ref{E2-4}). $F\in\textbf{B}_{R,T}^s$ is given.
Assume that (\ref{E1-2}) and (\ref{E2-5})-(\ref{E2-5-2}) hold.
Then for $s\geq2$, there holds
\begin{eqnarray}\label{E2-32}
|||u|||_{s,T}+|||r|||_{s,T}&\leq&c_{19}(||u_0||_{s+1,T}+||u_1||_{s+1,T}+||r_0||_{s+1,T}+||r_1||_{s+1,T}\nonumber\\
&&+|||h^A|||_{s,T}+|||h^d|||_{s,T}),
\end{eqnarray}
where $c_{19}$ depends on $T$.
\end{lemma}
\begin{proof}
The proof is based on the induction. For $s=1$, Lemma 3.3 gives the case by choosing a suitable $\lambda$. We assume that (\ref{E2-31}) holds for all $1\leq j\leq s$, and we prove that (\ref{E2-31}) holds for $s+1$. Since (\ref{E2-29-1})-(\ref{E2-30-1}) has the same structure with (\ref{E2-1})-(\ref{E2-2}), the result in Lemma 3.3 can be used directly.
Note that the vector $D^su^A$ and $D^sr^d$ satisfy (\ref{E2-29-1})-(\ref{E2-30-1}). By (\ref{E2-14}), we have
\begin{eqnarray}\label{E2-29}
&&(\lambda-c_{20})\int_0^T\int_{B_1(0)}e^{-\lambda t}[(\partial_i^su^A_t)^2+(\partial_i^{s+1}u^A)^2+(\partial_i^su^A)^2\nonumber\\
&&+(\partial_i^sr_t^d)^2+(\partial_i^sr^d)^2]dxdt\nonumber\\
&\leq&c_{21}\int_{B_1(0)}e^{-\lambda t}[(\partial_i^su^A_0)^2+(\partial_i^{s+1}u^A_0)^2+(\partial_i^{s}u^A_1)^2+(\partial_i^{s+1}u^A_1)^2\nonumber\\
&&+(\partial_i^sr_0^d)^2+(\partial_i^{s+1}r^d_0)^2+(\partial_i^sr_1^d)^2+(\partial_i^{s+1}r^d_1)^2]dx\nonumber\\
&&+c_{22}\sum_{|\gamma|=s}\int_0^T\int_{B_1(0)}e^{-\lambda t}(|h_{\gamma}^A|^2+|\partial_ih^A_{\gamma}|^2+|h^{d}_{\gamma}|^2+|\partial_ih^d_{\gamma}|^2)dxdt.~~~
\end{eqnarray}
By (\ref{E2-26-0}), (\ref{E2-27R0})-(\ref{E2-27R1}), we drive
\begin{eqnarray}\label{E2-29R0}
&&|h_{\gamma}^A|^2+|\partial_ih^A_{\gamma}|^2+|h^{d}_{\gamma}|^2+|\partial_ih^d_{\gamma}|^2\nonumber\\
&\leq&|\partial_i^{s}h^{A}|^2+|\partial_i^{s+1}h^{A}|^2+|\partial_i^{s}h^{d}|^2+|\partial_i^{s+1}h^{d}|^2\nonumber\\
&&+c_{23}(|\partial_i^{s+1}u^A|^2+|\partial_i^{s}u^A|+|\partial_i^{s-1}u_t^A|+|\partial_i^sr^d|+|\partial_i^{s-1}r_t^d|).
\end{eqnarray}
So choosing a $\lambda\geq c_{20}+c_{24}+1$ large enough, it follows from (\ref{E2-29}) and (\ref{E2-29R0}) that
\begin{eqnarray*}
&&(\lambda-c_{20}-c_{24})\int_0^T\int_{B_1(0)}e^{-\lambda t}[(\partial_i^su^A_t)^2+(\partial_i^{s+1}u^A)^2+(\partial_i^su^A)^2\nonumber\\
&&+(\partial_i^sr_t^d)^2+(\partial_i^sr^d)^2]dxdt\nonumber\\
&\leq&c_{21}\int_{B_1(0)}e^{-\lambda t}[(\partial_i^su^A_0)^2+(\partial_i^{s+1}u^A_0)^2+(\partial_i^{s}u^A_1)^2+(\partial_i^{s+1}u^A_1)^2\nonumber\\
&&+(\partial_i^sr_0^d)^2+(\partial_i^{s+1}r^d_0)^2+(\partial_i^sr_1^d)^2+(\partial_i^{s+1}r^d_1)^2]dx\nonumber\\
&&+c_{25}\int_0^T\int_{B_1(0)}e^{-\lambda t}(|\partial_i^{s}h^{A}|^2+|\partial_i^{s+1}h^{A}|^2+|\partial_i^{s}h^{d}|^2+|\partial_i^{s+1}h^{d}|^2)dxdt,
\end{eqnarray*}
which implies that
\begin{eqnarray}\label{E2-30}
&&\|\partial_i^s\partial_tu^A\|_{\textbf{L}^2(B_1(0))}+\|\partial_i^{s+1}u^A\|_{\textbf{L}^2(B_1(0))}+\|\partial_i^s\partial_tr^d\|_{\textbf{L}^2(B_1(0))}+\|\partial_i^{s}r^d\|_{\textbf{L}^2(B_1(0))}\nonumber\\
&\leq& c_{26}(\|u^A_0\|_{\textbf{H}^{s+1}(B_1(0))}+\|u^A_1\|_{\textbf{H}^{s+1}(B_1(0))}\nonumber\\
&&+\|r^d_0\|_{\textbf{H}^{s+1}(B_1(0))}+\|r^d_1\|_{\textbf{H}^{s+1}(B_1(0))}+\|h^A\|_{\textbf{H}^{s+1}(B_1(0))}+\|h^d\|_{\textbf{H}^{s+1}(B_1(0))}).~~~~~~~~~
\end{eqnarray}
Furthermore, we can apply $\partial_t^j\partial_i^{s+1-j}$ to both sides of (\ref{E2-1})-(\ref{E2-2}) for $2\leq j\leq s+1$, then deriving a similar estimate with (\ref{E2-30}). We conclude that
\begin{eqnarray}\label{E2-31}
&&\sum_{j=0}^{s+1}(\|\partial_t^j\partial_i^{s+1-j}u^A\|_{\textbf{L}^2(B_1(0))}+\|\partial_t^j\partial_i^{s+1-j}r^d\|_{\textbf{L}^2(B_1(0))})\nonumber\\
&\leq& c_{27}(\|u^A_0\|_{\textbf{H}^{s+2}(B_1(0))}+\|u^A_1\|_{\textbf{H}^{s+2}(B_1(0))}+\|r^d_0\|_{\textbf{H}^{s+2}(B_1(0))}\nonumber\\
&&+\|r^d_1\|_{\textbf{H}^{s+2}(B_1(0))}+\|h^A\|_{\textbf{H}^{s+2}(B_1(0))}+\|h^d\|_{\textbf{H}^{s+2}(B_1(0))}).
\end{eqnarray}
Summing up the estimate (\ref{E2-31}) over all coordinate charts, and we notice that $\|u\|_s$ is equivalent to $\|u\|_{s}+\|\partial\partial_tu\|_{s}$, hence (\ref{E2-32}) can be derived by (\ref{E2-31}).
This completes the proof.
\end{proof}
Note that $h=u^A\nu_A+r^d\tau_d$ and $W=h^A\nu_A+h^d\tau_d$. So by Lemma 3.5, direct computation gives the following result.
\begin{lemma}
Let $h$ be a smooth solution to weakly linear hyperbolic system (\ref{E2-1})-(\ref{E2-2}) with initial data $h(0)=h_0$ and $\partial_th(0)=h_1$. $F\in\textbf{B}_{R,T}^s$ is given.
Assume that (\ref{E1-2}) and (\ref{E2-5})-(\ref{E2-5-2}) hold.
Then for $s\geq\max\{2,\frac{n}{2}\}$, there holds
\begin{eqnarray}\label{E2-33}
|||h|||_{s,T}&\leq&c_{28}(||h_0||_{s+1,T}+||h_1||_{s+1,T}+|||W|||_{s,T}).
\end{eqnarray}
\end{lemma}

\textbf{Proof of Proposition 3.1.} The proof is based on a standard fixed point iteration. Let $h=u^{A}\nu_A+r^d\tau_d$ and $W=h^A\nu_A+h^d\tau_d$. We consider the following approximation system of (\ref{E2-1})-(\ref{E2-2}) as
\begin{eqnarray*}
&&\partial_{tt}u^A_{(m+1)}-v(F)L^Au^A_{(m+1)}-v(F)N^Ar_{(m)}-v(F)Q^Au^A_{(m)}=h^A,\\
&&\partial_{tt}r^d_{(m+1)}-v(F)M^dr_{(m)}-v(F)P^du_{(m)}=h^d,
\end{eqnarray*}
with initial data
\begin{eqnarray*}
u^A_{m+1}(0)&=&u^A_0,~~r^d_{m+1}(0)=r^d_0,\\
\partial_tu^A_{m+1}(0)&=&u^A_1=<h_1,\nu_B(0)>\nu^{AB}(0)-u^C_0<\partial_t\nu_C(0),\nu_B(0)>\nu^{AB}(0)\\
&&-r^d_0<\partial_t\tau_d(0),\nu_B(0)>\nu^{AB}(0),\\
\partial_tr^d_{m+1}(0)&=&r^d_1=<h_1,\tau_l(0)>\tau^{ld}(0)-u^A_0<\partial_t\nu_A(0),\tau_l(0)>\tau^{lk}(0)\\
&&-r^m_0<\partial_t\tau_m(0),\tau_l(0)>\nu^{ld}(0).
\end{eqnarray*}
Above system is a linear wave equation coupled with an ODEs. By a standard fixed point iteration and similar estimates in Lemma 3.5, we can prove that it has a local smooth solution on $[0,T]$. For the existence of solution for general linear wave equation, one can see \cite{Notz1,Notz,Shat,Sogge} for more details.

\section{Proof of Theorem 1.1}\setcounter{equation}{0}

In this section, we construct a solution $F(t,x)$ on $[0,\frac{T}{\sqrt{\varepsilon}})$ of equation (\ref{E1-1}) by a Nash-Moser iteration scheme.
Rescaling in (\ref{E1-1}) amplitude and time as
\begin{eqnarray*}
F(t,x)\mapsto\varepsilon F(\sqrt{\varepsilon}t,x),~~\varepsilon>0,
\end{eqnarray*}
we obtain the following system
\begin{eqnarray}\label{E3-1}
\partial_{tt}F
&=&\varepsilon^{n-1}\frac{d\mu_t}{d\mu}v(\varepsilon F)(-H(\varepsilon F)+\varepsilon\varphi(\varepsilon F) F\nonumber\\
&&-\varepsilon^2\varphi(\varepsilon F)\nabla(\frac{|F|^2}{2})+\frac{\rho}{Vol(\varepsilon F)})\nu,
\end{eqnarray}
with initial data
\begin{eqnarray}\label{E3-2}
F(0,\cdot)=F_0,~~\partial_tF(0,\cdot)=F_1.
\end{eqnarray}
Introduce an auxiliary function
\begin{eqnarray}\label{E3-3}
\overline{F}=F-F_0-tF_1,
\end{eqnarray}
then the initial value problem (\ref{E3-1})-(\ref{E3-2}) is equivalent to
\begin{eqnarray}\label{E3-4}
\partial_{tt}\overline{F}
&=&\varepsilon^{n-1}\frac{d\mu_t}{d\mu}v(\varepsilon (\overline{F}+F_0+tF_1))(-H(\varepsilon (\overline{F}+F_0+tF_1))\nonumber\\
&&+\varepsilon(\overline{F}+F_0+tF_1)\varphi(\varepsilon (\overline{F}+F_0+tF_1))\nonumber\\
&&-\varepsilon^2\varphi(\varepsilon(\overline{F}+F_0+tF_1))\nabla(\frac{|\overline{F}+F_0+tF_1|^2}{2})\nonumber\\
&&+\frac{\rho}{Vol(\varepsilon (\overline{F}+F_0+tF_1))})\nu,
\end{eqnarray}
with zero initial data
\begin{eqnarray}\label{E3-5}
\overline{F}(0,\cdot)=0,~~\partial_t\overline{F}(0,\cdot)=0.
\end{eqnarray}
We define
\begin{eqnarray}\label{E3-4-1}
\mathcal{N}(\overline{F})&:=&\partial_{tt}\overline{F}
-\varepsilon^{n-1}\frac{d\mu_t}{d\mu}v(\varepsilon (\overline{F}+F_0+tF_1))(-H(\varepsilon (\overline{F}+F_0+tF_1))\nonumber\\
&&+\varepsilon(\overline{F}+F_0+tF_1)\varphi(\varepsilon (\overline{F}+F_0+tF_1))\nonumber\\
&&-\varepsilon^2\varphi(\varepsilon(\overline{F}+F_0+tF_1))\nabla(\frac{|\overline{F}+F_0+tF_1|^2}{2})\nonumber\\
&&+\frac{\rho}{Vol(\varepsilon (\overline{F}+F_0+tF_1))})\nu,
\end{eqnarray}
then our target is to construct a $\overline{F}(\sqrt{\varepsilon}t,x)\in\textbf{B}_{R,T}^s$ on the time interval $[0,T]$ such that $\mathcal{N}(\overline{F})=0$. Here $\sqrt{\varepsilon}t\in[0,T]$.
Thus we obtain the existence of smooth solution for (\ref{E3-4}) with initial data $(\ref{E3-5})$.
In fact, we treat the initial value problem (\ref{E3-4})-(\ref{E3-5}) iteratively as a small perturbation of the initial value problem for a weakly linear
hyperbolic system.
Linearizing nonlinear equation (\ref{E3-4}), we obtain
the linearized operator
\begin{eqnarray}\label{E3-6}
&&\partial_{\overline{F}}\mathcal{N}(\overline{F})h:=\partial_{tt}h-\varepsilon^{n-1}\frac{d\mu_t}{d\mu}v(\varepsilon (\overline{F}+F_0+tF_1))[(\triangle<h,\nu>+|\tilde{h}|^2<h,\nu>\nonumber\\
&&-<\nabla H,h^{\top}>)+\varepsilon h\varphi(\varepsilon (\overline{F}+F_0+tF_1)+\varepsilon (\overline{F}+F_0+tF_1)(\partial_{\overline{F}}\varphi)h)\nonumber\\
&&-\varepsilon^2(\partial_{\overline{F}}\varphi)h\nabla(\frac{|\overline{F}+F_0+tF_1|^2}{2})
-\varepsilon^2\varphi(\varepsilon(\overline{F}+F_0+tF_1))\nabla<\overline{F}+F_0+tF_1,h>\nonumber\\
&&+\frac{\rho}{Vol(\varepsilon (\overline{F}+F_0+tF_1))^2}\int_{\Sigma}<h,\nu>d\mu_t]\nu\nonumber\\
&&-\varepsilon^{n-1}\frac{d\mu_t}{d\mu}v(\varepsilon (\overline{F}+F_0+tF_1))[-H(\varepsilon (\overline{F}+F_0+tF_1))\nonumber\\
&&+\varepsilon(\overline{F}+F_0+tF_1)\varphi(\varepsilon (\overline{F}+F_0+tF_1))\nonumber\\
&&-\varepsilon^2\varphi(\varepsilon(\overline{F}+F_0+tF_1))\nabla(\frac{|\overline{F}+F_0+tF_1|^2}{2})\nonumber\\
&&+\frac{\rho}{Vol(\varepsilon (\overline{F}+F_0+tF_1))}](div h^{\top}+<h,\nu>H)\nu\nonumber\\
&&+\varepsilon^{n}\frac{d\mu_t}{d\mu}v(\varepsilon (\overline{F}+F_0+tF_1))[-H(\varepsilon (\overline{F}+F_0+tF_1))\nonumber\\
&&+\varepsilon(\overline{F}+F_0+tF_1)\varphi(\varepsilon (\overline{F}+F_0+tF_1))-\varepsilon^2\varphi(\varepsilon(\overline{F}+F_0+tF_1))\nabla(\frac{|\overline{F}+F_0+tF_1|^2}{2})\nonumber\\
&&+\frac{\rho}{Vol(\varepsilon (\overline{F}+F_0+tF_1))}]\varphi(\varepsilon(\overline{F}+F_0+tF_1))<|\overline{F}+F_0+tF_1|,h>\nu\nonumber\\
&&-\varepsilon^{n-1}\frac{d\mu_t}{d\mu}v(\varepsilon (\overline{F}+F_0+tF_1))[-H(\varepsilon (\overline{F}+F_0+tF_1))\nonumber\\
&&+\varepsilon(\overline{F}+F_0+tF_1)\varphi(\varepsilon (\overline{F}+F_0+tF_1))\nonumber\\
&&-\varepsilon^2\varphi(\varepsilon(\overline{F}+F_0+tF_1))\nabla(\frac{|\overline{F}+F_0+tF_1|^2}{2})\nonumber\\
&&+\frac{\rho}{Vol(\varepsilon (\overline{F}+F_0+tF_1))}](\nabla<h,\nu>-\varepsilon^2<h,\partial_l\overline{F}>\tilde{h}^{dl}\partial_d\overline{F}),
\end{eqnarray}
where $\tilde{h}_{ij}$ denote the second fundamental form, $\partial_{\overline{F}}$ denotes the Fr\'{e}chet derivative to $\overline{F}$ and we use the following relations
\begin{eqnarray*}
&&-\partial_{\overline{F}}H=\triangle<h,\nu>+|\tilde{h}|^2<h,\nu>-<\nabla H,h^{\top}>,\\
&&\partial_{\overline{F}}d\mu_t=(div h^{\top}+<h,\nu>H)d\mu_t,\\
&&\partial_{\overline{F}}\nu=\nabla<h,\nu>-<h,\partial_l\overline{F}>\tilde{h}^{ld}\partial_d\overline{F},\\
&&\partial_{\overline{F}}v(\varepsilon(\overline{F}+F_0+tF_1))=-\varepsilon\varphi(\tilde{s})v(\varepsilon(\overline{F}+F_0+tF_1))<|\overline{F}+F_0+tF_1|,h>.
\end{eqnarray*}
For the nonlinear term, by (\ref{E3-4}) and (\ref{E3-6}), direct computations show that
\begin{eqnarray}\label{E3-7}
\mathcal{R}(h):=\mathcal{N}(\overline{F}+h)-\mathcal{N}(\overline{F})-\partial_{\overline{F}}\mathcal{N}(\overline{F})h.
\end{eqnarray}
Since the exact form of nonlinear term (\ref{E3-7}) is very complicated, here we does not write it down. In fact,
we notice that the solution of (\ref{E3-4}) is to be constructed in $\textbf{B}_{R,T}^s$ and $\overline{F}\in\textbf{B}_{R,T}^s$, so we have
\begin{eqnarray*}
&&|||\overline{F}|||_{s,T}\leq R<1,\\
&&|||h|||_{s,T}^q\leq|||h|||_{s,T}^2,~~\forall q\geq2.
\end{eqnarray*}
Thus we only need to obtain the lowest exponent of $h$ for convenience. Then by the product inequality, we can obtain the following result, i.e. Lemma 4.1.
Note that
\begin{eqnarray*}
&&v(s)=\exp(-\frac{n}{2}\int_1^s\frac{\eta(w)}{w}dw),~~d\mu_t=\sqrt{\det(g_{ij})},\\
&&\nu=\frac{\partial_1F\times\ldots\times\partial_nF}{|\partial_1F\times\ldots\times\partial_nF|},~~Vol(F)=Vol_0+\int_0^t<\partial_tF,\nu>d\mu_tds.
\end{eqnarray*}
Direct computation shows that the lowest exponent of $\overline{F}$ is $n+1$ in (\ref{E3-4}), which is determined by the lowest exponent of $\overline{F}$ in $d\mu_t$. This lowest exponent of $h$ in (\ref{E3-7}) equals to the lowest exponent of $\overline{F}$ in (\ref{E3-4}), i.e. $n+1$.
\begin{lemma}
Let $\overline{F},~h\in\textbf{B}_{R,T}^s$.
For any $s>\max\{2,\frac{n}{2}\}$, there holds
\begin{eqnarray}\label{E3-8}
|||\mathcal{R}(h)|||_{s,T}\leq c_{29}|||h|||_{s+2,T}^{n+1},
\end{eqnarray}
where $c_{29}$ depends on the constant $R$.
\end{lemma}
In order to prove the existence of smooth solution for the linearized equation, we can follow \cite{Notz} to decompose $h$ and the external force $W(t,x)$ into the normal and tangential parts.
Then direct computation shows that the linearized equation has the same form as the weakly linear hyperbolic system (\ref{E2-1})-(\ref{E2-2}). So by Proposition 3.1 and Lemma 3.5, we obtain the following result.
\begin{lemma}
Let $\overline{F}\in\textbf{B}_{R,T}^s$. For any $s>\max\{2,\frac{n}{2}\}$, there exists a function $h\in\textbf{B}_{R,T}^s$ such that
\begin{eqnarray*}
&&\partial_{\overline{F}}\mathcal{N}(\overline{F})h=W(t,x),\\
&&h(0,\cdot)=0,~~\partial_th(0,\cdot)=0.
\end{eqnarray*}
Moreover, there holds
\begin{eqnarray}\label{E3-9}
|||h|||_{s,T}&\leq&c_{30}|||W|||_{s,T}.
\end{eqnarray}
\end{lemma}
Next we introduce the smooth truncation function, one can see \cite{Schwartz} for more details.
Let $\Pi_{\theta}\in\textbf{C}^{\infty}(\textbf{R})$ such that
$\Pi_{\theta}=0$ for $\theta\leq0$ and $\Pi_{\theta}\longrightarrow I$ for $\theta\longrightarrow\infty$.
we introduce a family of smooth functions $S(\theta')$ with $S(\theta')=0$ for $\theta'\leq0$ and $S(\theta')=1$ for $\theta'\geq1$.
For $W\in\textbf{C}_T^s$, we define
\begin{eqnarray*}
\Pi_{|x'|-|x|}W(t,x)=S(|x'|-|x|)W(t,x),~~x,x'\in \textbf{R}^n.
\end{eqnarray*}
For $l=0,1,2,\ldots,$, by setting
\begin{eqnarray}\label{E3-10}
|x'|:=N_l=2^l,
\end{eqnarray}
then by the Fourier transform is an isomorphism between $\textbf{L}^2_s$ and $\textbf{H}^s$, it is directly to check that
\begin{eqnarray}\label{E3-11}
&&\|\Pi_{|x'|-|x|}W\|_{\textbf{H}^{s_1}}\leq c_{s_1,s_2}N_l^{s_1-s_2}\|W\|_{\textbf{H}^{s_2}},~~\forall~s_1\geq s_2\geq0,\\
&&\|\Pi_{|x'|-|x|}W-W\|_{\textbf{H}^{s_1}}\leq c_{s_1,s_2}N_l^{s_1-s_2}\|W\|_{\textbf{H}^{s_2}},~~\forall~0\leq s_1\leq s_2.\nonumber
\end{eqnarray}
For convenience, we denote $\Pi_{N_l-|x|}$ by $\Pi_{N_l}$. We approximate system (\ref{E3-4-1}), and get
the following approximation system
\begin{eqnarray}\label{E3-12}
\mathcal{G}_{N_l}(\overline{F})&:=&\partial_{tt}\overline{F}
-\varepsilon^{n-1}\Pi_{N_l}\frac{d\mu_t}{d\mu}v(\varepsilon (\overline{F}+F_0+tF_1))[-H(\varepsilon (\overline{F}+F_0+tF_1))\nonumber\\
&&+\varepsilon(\overline{F}+F_0+tF_1)\varphi(\varepsilon (\overline{F}+F_0+tF_1))\nonumber\\
&&-\varepsilon^2\varphi(\varepsilon(\overline{F}+F_0+tF_1))\nabla(\frac{|\overline{F}+F_0+tF_1|^2}{2})\nonumber\\
&&+\frac{\rho}{Vol(\varepsilon (\overline{F}+F_0+tF_1))}]\nu,
\end{eqnarray}
where $m=0,1,\ldots,l,\ldots$.

The following Lemma is to construct the ``$l$ th'' step approximation solution.
\begin{lemma}
There exists a
solution $h^{l+1}$ of the initial value problem
\begin{eqnarray*}
&&E^{l}+\partial_{\overline{F}}\mathcal{N}(\overline{F})h^{l+1}=0,\\
&&h^{l+1}(0,x)=0,~~\partial_th^{l+1}(0,x)=0,
\end{eqnarray*}
where $E^{l}$ satisfies
\begin{eqnarray*}
E^{l}=\mathcal{G}_{N_{l}}(\overline{F}^l).
\end{eqnarray*}
Moreover, for $s>\max\{2,\frac{n}{2}\}$, it holds
\begin{eqnarray}\label{E3-13}
|||h^{l+1}|||_{s,T}\leq c_{31}|||E^l|||_{s,T}.
\end{eqnarray}
\end{lemma}
\begin{proof}
Assume that a suitable ''$0$th step'' approximation solution of (\ref{E3-12}) has been chosen, which is $W^0\neq0$. The ``$l$th step'' approximation solution is denoted by
\begin{eqnarray*}
\overline{F}^l=\sum_{i=0}^lh^i.
\end{eqnarray*}
Define
\begin{eqnarray}\label{E3-14}
E^l&=&\partial_{tt}\overline{F}^l
-\varepsilon^{n-1}\Pi_{N_l}\frac{d\mu_t}{d\mu}v(\varepsilon (\overline{F}^l+F_0+tF_1))[-H(\varepsilon (\overline{F}^l+F_0+tF_1))\nonumber\\
&&+\varepsilon(\overline{F}^l+F_0+tF_1)\varphi(\varepsilon (\overline{F}^l+F_0+tF_1))\nonumber\\
&&-\varepsilon^2\varphi(\varepsilon(\overline{F}^l+F_0+tF_1))\nabla(\frac{|\overline{F}^l+F_0+tF_1|^2}{2})\nonumber\\
&&+\frac{\rho}{Vol(\varepsilon (\overline{F}^l+F_0+tF_1))}]\nu
\end{eqnarray}
Then we plan to find the ``$l$th step'' approximation solution $\overline{F}^{l+1}$. By (\ref{E3-12}), we have
\begin{eqnarray}\label{E3-15}
&&\mathcal{G}(\overline{F}^l+h^{l+1})\nonumber\\
&=&\partial_{tt}\overline{F}^l+\partial_{tt}h^{l+1}
-\varepsilon^{n-1}\Pi_{N_{l+1}}\frac{d\mu_t}{d\mu}v(\varepsilon (\overline{F}^l+h^{l+1}+F_0+tF_1))\nonumber\\
&&\times[-H(\varepsilon (\overline{F}^l+h^{l+1}+F_0+tF_1))\nonumber\\
&&+\varepsilon(\overline{F}^l+F_0+tF_1)\varphi(\varepsilon (\overline{F}^l+h^{l+1}+F_0+tF_1))\nonumber\\
&&-\varepsilon^2\varphi(\varepsilon(\overline{F}^l+F_0+tF_1))\nabla(\frac{|\overline{F}^l+h^{l+1}+F_0+tF_1|^2}{2})\nonumber\\
&&+\frac{\rho}{Vol(\varepsilon (\overline{F}^l+h^{l+1}+F_0+tF_1))}]\nu\nonumber\\
&=&E^l+\partial_{\overline{F}}\mathcal{N}(\overline{F})h^{l+1}+R(h^{l+1}),
\end{eqnarray}
where
\begin{eqnarray}\label{E3-16}
R(h^{l+1})=-\varepsilon^{n-1}\Pi_{N_{l+1}}\left(\mathcal{N}(\overline{F}^l+h^{l+1})-\Pi_{N_l}\mathcal{N}(\overline{F}^l)-\partial_{\overline{F}^l}\mathcal{N}(\overline{F}^l)h^{l+1}\right).~~~
\end{eqnarray}
By Lemma 4.2, there exists a solution $h^{l+1}$ of
\begin{eqnarray*}
&&E^l+\partial_{\overline{F}}\mathcal{N}(\overline{F})h^{l+1}=0,\\
&&h^{l+1}(0,\cdot)=0,~~\partial_th^{l+1}(0,\cdot)=0.
\end{eqnarray*}
A similar estimate with (\ref{E3-9}) is derived as
\begin{eqnarray*}
|||h^{l+1}|||_{s,T}\leq c_{31}|||E^l|||_{s,T}.
\end{eqnarray*}
Furthermore, one can know from (\ref{E3-15}) and (\ref{E3-16}) that
\begin{eqnarray}\label{E3-17}
E^{l+1}=R(h^{l+1}).
\end{eqnarray}
\end{proof}

For $\max\{2,\frac{n}{2}\}\leq \bar{s}<s_0\leq s$, set
\begin{eqnarray}\label{E3-18}
&&s_l:=\bar{s}+\frac{s-\bar{s}}{2^l},\\
\label{E3-19}
&&\alpha_{l+1}:=s_l-s_{l+1}=\frac{s-\bar{s}}{2^{l+1}}.
\end{eqnarray}
By (\ref{E3-18})-(\ref{E3-19}), it follows that
\begin{eqnarray*}
s_0>s_1>\ldots>s_l>s_{l+1}>\ldots.
\end{eqnarray*}

\begin{theorem}
Equation (\ref{E3-1}) with initial data (\ref{E3-2})
has a solution
\begin{eqnarray}\label{E3-20}
F=\overline{F}_{\infty}+F_0+F_1t,
\end{eqnarray}
where $\overline{F}_{\infty}$ has the form
\begin{eqnarray*}
\overline{F}_{\infty}=\sum_{i=0}^{\infty}h_i\in\textbf{C}_{T}^{\bar{s}}.
\end{eqnarray*}
\end{theorem}
\begin{proof}
The proof is based on the induction. For any $l=0,1,2,\ldots$, we claim that there exists a constant $0<d<1$ such that
\begin{eqnarray}\label{E3-14R0}
&&|||h^{l+1}|||_{s_{l+1},T}<d^{2^{l}}<1,\\
\label{E3-14R1}
&&|||E^{l+1}|||_{s_{l+1}}\leq d^{2^{l+1}},\\
\label{E3-14R2}
&&\overline{F}^{l+1}\in\textbf{B}_{R,T}^{s}.
\end{eqnarray}
We choose a fixed sufficient small $\overline{F}^0>0$ such that
\begin{eqnarray}\label{E3-21}
|||\overline{F}^0|||_{s_0}\ll1,~~|||E^0|||_{s_0}\ll1.
\end{eqnarray}
For the case $l=0$, by (\ref{E3-13}), we have
\begin{eqnarray}\label{E3-22}
|||h^{1}|||_{s_{1},T}\leq c_{31}|||E^0|||_{s_{1},T}\leq c_{31}|||E^0|||_{s_{0},T}<1.
\end{eqnarray}
It follows from (\ref{E3-8}), (\ref{E3-11}), (\ref{E3-16}) and (\ref{E3-17}) that
\begin{eqnarray}
\label{E3-23}
|||E^{1}|||_{s_{1},T}&\leq&|||R(h^{1})|||_{s_{1},T}\nonumber\\
&\leq& c_{32}\varepsilon^{n-1}N_1^{n+1}|||h^1|||_{s_1,T}^{n+1}\nonumber\\
&\leq& c_{33}(2\varepsilon|||E^0|||_{s_{0},T})^{n+1}.
\end{eqnarray}
It is obviously to see that (\ref{E3-22})-(\ref{E3-23}) gives $(\ref{E3-14R0})_{l=0}-(\ref{E3-14R1})_{l=0}$ by choosing suitable small $\varepsilon>0$. So we get $\overline{F}^1\in\textbf{B}_{R,T}^{s_1}$.

Assume that (\ref{E3-14R0})-(\ref{E3-14R2}) holds for $1\leq i\leq l$, i.e.
\begin{eqnarray}\label{E3-24}
&&|||h^{i}|||_{s_{i},T}<1,\nonumber\\
\label{E3-25}
&&|||E^{i}|||_{s_{i}}\leq d^{2^{i}},\\
\label{E3-26}
&&\overline{F}^{i}\in\textbf{B}_{R,T}^{s^i}\nonumber.
\end{eqnarray}
Now we prove that (\ref{E3-14R0})-(\ref{E3-14R2}) holds for $l+1$.
From (\ref{E3-13}) and (\ref{E3-25}), we have
\begin{eqnarray}\label{E3-14}
|||h^{l+1}|||_{s_{l+1},T}\leq c_{31}|||E^l|||_{s_{l},T}<c_{31} d^{2^{i}}<1.
\end{eqnarray}
It follows from (\ref{E3-8}), (\ref{E3-11}), (\ref{E3-13}), (\ref{E3-16}) and (\ref{E3-17}) that
\begin{eqnarray}
\label{E3-27}
|||E^{l+1}|||_{s_{l+1},T}&\leq&|||R(h^{l+1})|||_{s_{l+1},T}\nonumber\\
&\leq&c_{34}\varepsilon^{n-1}N_{l+1}^{n+1}|||h^{l+1}|||_{s_{l+1},T}^{n+1}\nonumber\\
&\leq&c_{34}\varepsilon^{n-1}N_{l+1}^{n+1}|||E^l|||_{s_{l},T}^{n+1}\nonumber\\
&\leq&c_{34}\varepsilon^{n-1}N_{l+1}^{n+1}N_l^{(n+1)^2}|||E^{l-1}|||_{s_{l-1},T}^{(n+1)^2}\nonumber\\
&\leq&\ldots\nonumber\\
&\leq&c_{35}(16\varepsilon|||E^{0}|||_{s_{0},T})^{(n+1)^{l+1}}.
\end{eqnarray}
We can choose a fixed sufficient small $\varepsilon>0$ such that
\begin{eqnarray*}
0<16\varepsilon|||E^{0}|||_{s_{0},T}<1.
\end{eqnarray*}
Thus we conclude that (\ref{E3-14R0})-(\ref{E3-14R1}) holds. Note that $\overline{F}^l=\sum_{i=0}^lh^i$. So (\ref{E3-14R0}) gives (\ref{E3-14R2}).

Therefore, we derive
\begin{eqnarray*}
\lim_{l\longrightarrow\infty}|||E^l|||_{l,T}=0,
\end{eqnarray*}
which implies that system (\ref{E3-4}) with zero initial data has a solution
\begin{eqnarray*}
\overline{F}_{\infty}=\sum_{i=0}^{\infty}h_i\in\textbf{C}_{T}^{\bar{s}}.
\end{eqnarray*}
At last, by (\ref{E3-3}) we obtain the solution of system (\ref{E3-1}) with initial data (\ref{E3-2}) has a solution
\begin{eqnarray*}
F=\overline{F}_{\infty}+F_0+F_1t.
\end{eqnarray*}
\end{proof}

\section{Proof of Theorem 1.2}\setcounter{equation}{0}

This section is to study the stability problem of (\ref{E1-1}) in $\textbf{B}_{R,T}^{s_1}$.
Assume that there exist two solutions $\bar{F}$ and $\tilde{F}$ of (\ref{E3-1}) with corresponding to two different initial data $(\bar{F}_0,\bar{F}_1)$ and $(\tilde{F}_0,\tilde{F}_1)$. Let $F=\bar{F}-\tilde{F}$. Then we have
\begin{eqnarray}\label{E4-1}
\partial_{tt}F&=&\varepsilon^{n-1}\frac{d\mu_t(\bar{F})-d\mu_t(\tilde{F})}{d\mu}v(\varepsilon\bar{F})\nonumber\\
&&\times\left(-H(\varepsilon\bar{F})+\varepsilon\varphi(\varepsilon\bar{F}) \bar{F}-\varepsilon^2\varphi(\varepsilon\bar{F})\nabla(\frac{|\bar{F}|^2}{2})+\frac{\rho}{Vol(\varepsilon\bar{F})}\right)\nu(\bar{F})\nonumber\\
&&+\varepsilon^{n-1}\frac{d\mu_t(\tilde{F})}{d\mu}(v(\varepsilon\bar{F})-v(\varepsilon\tilde{F}))\nonumber\\
&&\times\left(-H(\varepsilon\bar{F})+\varepsilon\varphi(\varepsilon\bar{F}) \bar{F}-\varepsilon^2\varphi(\varepsilon\bar{F})\nabla(\frac{|\bar{F}|^2}{2})+\frac{\rho}{Vol(\varepsilon\bar{F})}\right)\nu(\bar{F})\nonumber\\
&&+\varepsilon^{n-1}\frac{d\mu_t(\tilde{F})}{d\mu}v(\varepsilon\tilde{F})\nonumber\\
&&\times\left(-H(\varepsilon\bar{F})+\varepsilon\varphi(\varepsilon\bar{F}) \bar{F}-\varepsilon^2\varphi(\varepsilon\bar{F})\nabla(\frac{|\bar{F}|^2}{2})+\frac{\rho}{Vol(\varepsilon\bar{F})}\right)(\nu(\bar{F})-\nu(\tilde{F}))\nonumber\\
&&+\varepsilon^{n-1}\frac{d\mu_t(\tilde{F})}{d\mu}v(\varepsilon\tilde{F})(-H(\varepsilon\bar{F})+H(\varepsilon\tilde{F})+\varepsilon(\varphi(\varepsilon\bar{F}) \bar{F}-\varphi(\varepsilon\tilde{F})\tilde{F})\nonumber\\
&&-\varepsilon^2(\varphi(\varepsilon\bar{F})\nabla(\frac{|\bar{F}|^2}{2})-\varphi(\varepsilon\tilde{F})\nabla(\frac{|\tilde{F}|^2}{2}))+\frac{\rho}{Vol(\varepsilon\bar{F})}-\frac{\rho}{Vol(\varepsilon\tilde{F})})\nu(\tilde{F}),
\end{eqnarray}
with initial data
\begin{eqnarray}\label{E4-2}
F_0=\bar{F}_0-\tilde{F}_0,~~F_1=\bar{F}_1-\tilde{F}_1.
\end{eqnarray}
Introduce an auxiliary function
\begin{eqnarray}\label{E4-3}
F^*=F-F_0-tF_1,
\end{eqnarray}
then the initial value problem (\ref{E4-1})-(\ref{E4-2}) is equivalent to a zero initial value problem. For convenience, we denote the right hand side of (\ref{E4-1}) by $\mathcal{N}(F,\bar{F},\tilde{F})$. Then equation (\ref{E4-1}) can be rewritten as
\begin{eqnarray}\label{E4-4}
\partial_{tt}F^*-\mathcal{N}(F^*,\bar{F},\tilde{F})=0.
\end{eqnarray}
We consider the approximation system of (\ref{E4-4}) as
\begin{eqnarray}\label{E4-5}
\mathcal{G}'(F^*):=\partial_{tt}F^*-\Pi_{N_l}\mathcal{N}(0,\bar{F},\tilde{F}).
\end{eqnarray}
Then using the similar computation process with (\ref{E3-15}), we have
\begin{eqnarray}\label{E4-6}
\mathcal{L}_{\varepsilon}(F^*)+E'(t,x)+R(F^*)-E'(t,x)=0,
\end{eqnarray}
where $E'(t,x)$ is a function which does not depends on $F^*$,
\begin{eqnarray}
&&\mathcal{L}_{\varepsilon}(F^*)=\partial_{tt}F^*-\Pi_{N_l}\partial_{F^*}\mathcal{N}(0,\bar{F},\tilde{F})F^*,\nonumber\\
\label{E4-7}
&&R(F^*)=\varepsilon^{n-1}\Pi_{N_l}(\mathcal{N}(F^*,\bar{F},\tilde{F})-\mathcal{N}(F,\bar{F},\tilde{F})-\partial_{F^*}\mathcal{N}(0,\bar{F},\tilde{F})F^*).
\end{eqnarray}
Direct computation shows that $\mathcal{L}_{\varepsilon}(F^*)$ is a weak hyperbolic system, which has the same structure as (\ref{E2-1})-(\ref{E2-2}). So
by Proposition 3.2, there exists a solution $F^*$ of
\begin{eqnarray*}
&&\mathcal{L}_{\varepsilon}(F^*)+E'(t,x)=0,\\
&&F^*_0=0,~~\partial_tF^*_1=0.
\end{eqnarray*}
A similar estimate with (\ref{E3-13}) is derived as
\begin{eqnarray*}
|||F^*|||_{s,T}\leq c_{36}|||E'|||_{s,T}.
\end{eqnarray*}
Since the lowest exponent of $F^*$ in (\ref{E4-7}) is $n+1$, so we have the following result.
\begin{lemma}
Let $\bar{F},~\tilde{F}\in\textbf{B}_{R,T}^s$.
For any $s>\max\{2,\frac{n}{2}\}$, there holds
\begin{eqnarray}\label{E4-8}
|||\mathcal{R}(F^*)|||_{s,T}\leq c_{37}|||F^*|||_{s+2,T}^{n+1},
\end{eqnarray}
where $c_{37}$ depends on the constant $R$.
\end{lemma}
Then by (\ref{E4-6}) and (\ref{E4-8}), we have
\begin{eqnarray*}
|||F^*|||_{s_l,T}&\leq& c_{36}|||E'|||_{s_l,T}\leq c_{36}|||R(F^*)|||_{s_l,T}\nonumber\\
&\leq&c_{38}\varepsilon^{n-1}N_{l}^{n+1}|||F^*|||_{s_l,T}^{n+1}\nonumber\\
&\leq&c_{39}\varepsilon^{n-1}N_{l}^{n+1}N_{l-1}^{(n+1)^2}|||F^*|||_{s_{l-2},T}^{(n+1)^2}\nonumber\\
&\leq&\ldots\nonumber\\
&\leq&c_{40}(8\varepsilon|||F^*|||_{s_{0},T})^{(n+1)^l},
\end{eqnarray*}
which combines with (\ref{E4-3}), we have
\begin{eqnarray*}
|||F|||_{s_l,T}\leq|||F_0|||_{s_l,T}+T|||F_1|||_{s_l,T}+c_{40}(8\varepsilon|||F^*|||_{s_{0},T})^{(n+1)^l}.
\end{eqnarray*}
Thus we obtain
\begin{eqnarray*}
|||\bar{F}-\tilde{F}|||_{s_l,T}\leq|||\bar{F}_0-\tilde{F}_0|||_{s_l,T}+T|||\bar{F}_1-\tilde{F}_1|||_{s_l,T}+c_{40}(8\varepsilon|||F^*|||_{s_{0},T})^{(n+1)^l},
\end{eqnarray*}
which implies the stability result for equation (\ref{E1-1}).

In particularly, if two solutions $\bar{F}$ and $\tilde{F}$ of (\ref{E3-1}) has the same initial data, i.e. $\bar{F}_0=\tilde{F}_0$ and $\bar{F}_1=\tilde{F}_1$. Then we
choose a suitable small $\varepsilon$ such that
\begin{eqnarray*}
0<8\varepsilon|||F^*|||_{s_{0},T}<1.
\end{eqnarray*}
Then we have
\begin{eqnarray*}
\lim_{l\longrightarrow\infty}|||\bar{F}-\tilde{F}|||_{s_l,T}=0,
\end{eqnarray*}
which gives the uniqueness of solution for equation (\ref{E1-1}) with the initial data $F_0$ and $F_1$.
This completes the proof of Theorem 1.2.\\

\section{Further Discussion}\setcounter{equation}{0}
Our main results can be extended to more general case, i.e. $(\mathcal{M},g)$ be an oriented smooth complete $n+1$ dimensional Riemannian manifold. This section will discuss this case shortly.
By Nash's embedding theorem we may embed $(\mathcal{M},g)$ isometrically into an Euclidean space $\textbf{R}^k$ by $f:\mathcal{M}\longrightarrow\textbf{R}^k$ for some large $k$. Then we use the similar method of \cite{Notz,Shat} to derive an extrinsic form of the Euler-Lagrange equation (\ref{E1-1}).

We denote $\Psi_{\mathcal{M}}$ as the closest point projection to $f(\mathcal{M})$, it is defined on a neighborhood
\begin{eqnarray*}
\tilde{\mathcal{M}}=\{x+y|x\in f(\mathcal{M}),~y\in(T_xf(\mathcal{M}))^{\bot},~~|y|<\delta(x)\},
\end{eqnarray*}
of $f(\mathcal{M})$ and is smooth there. Here $\delta(x)$ denotes a positive smooth function on $f(\mathcal{M})$.
Since the second fundamental form $\tilde{h}_{\alpha\beta}$ of $\mathcal{M}$ is normal to $f(\mathcal{M})$ and $f=\Psi_{\mathcal{M}}\circ f$, we get
\begin{eqnarray}\label{E4-1R}
\overline{\partial}_{\alpha}\overline{\partial}_{\beta}f-\overline{\Gamma}_{\alpha\beta}^{\gamma}\overline{\partial}_{\gamma}f
=D_AD_B\Psi_{\mathcal{M}}\overline{\partial}_{\beta}f^A\overline{\partial}_{\alpha}f^B,
\end{eqnarray}
where $D_A$ and $D_B$ denote the derivative in the direction of the canonical basis vector $e_A$ and $e_B$ in $\textbf{R}^k$, respectively.

Let $\tilde{F}:[0,T]\times\Sigma\longrightarrow\mathcal{M}$ and $F:[0,T]\times\Sigma\longrightarrow f(\mathcal{M})$ with $F=f(\tilde{F})$. Note that $f$ is an isometric embedding. Then by (\ref{E4-1R}), direct computation shows that
\begin{eqnarray*}
\overline{\square}_gF&-&\frac{d\mu_t}{d\mu}v(F)\frac{\rho}{vol(F)}\nu(F)-D_AD_B\Psi_{\mathcal{M}}(F)(\partial_tF^A\partial_tF^B-\frac{d\mu_t}{d\mu}v(F)g^{ij}\partial_iF^A\partial_jF^B)\nonumber\\
&-&D_f\frac{d\mu_t}{d\mu}v(F)\varphi(F)(f^{-1}\circ F-\nabla(\frac{|F|^2}{2}))\nu(F)\nonumber\\
&=&D_f(\overline{\nabla}_{\partial_t}\partial_t\tilde{F}-\frac{d\mu_t}{d\mu}v(-H(\tilde{F})+\varphi \tilde{F}-\varphi\nabla(\frac{|\tilde{F}|^2}{2})+\frac{\rho}{Vol(\tilde{F})})\nu(\tilde{F})),
\end{eqnarray*}
where
\begin{eqnarray*}
\overline{\square}_gF=\partial_{tt}F-\frac{d\mu_t}{d\mu}v(F)g^{ij}(\partial_i\partial_jF-\Gamma_{ij}^d\partial_dF).
\end{eqnarray*}
Hence $\tilde{F}$ is a solution of (\ref{E1-1}) if and only if $F$ is the solution of
\begin{eqnarray}\label{E4-2R}
\overline{\square}_gF&-&\frac{d\mu_t}{d\mu}v(F)\frac{\rho}{vol(F)}\nu(F)-D_f\frac{d\mu_t}{d\mu}v(F)\varphi(F)(f^{-1}\circ F-\nabla(\frac{|F|^2}{2}))\nu(F)\nonumber\\
&&-D_AD_B\Psi_{\mathcal{M}}(F)(\partial_tF^A\partial_tF^B-\frac{d\mu_t}{d\mu}v(F)g^{ij}\partial_iF^A\partial_jF^B)=0.
\end{eqnarray}
Next we follow the method of \cite{Notz} to extend equation (\ref{E4-2R}) for $F:[0,T]\times\Sigma\longrightarrow\tilde{\mathcal{M}}\subset\textbf{R}^k$ which do not necessarily map to $f(\mathcal{M})$. Let $\Psi_{\Omega_t^{\bot}}(F)$ be the projection onto the normal space of $\Omega_t=F(t,\Sigma)$. We notice that if $F$ is close enough in $\textbf{C}^1$ to a family of immersions that map to $f(\mathcal{M})$, then $\Psi_{\mathcal{M}}\circ F$ is also a family of immersions and $\nu(\Psi_{\mathcal{M}}\circ F)$ and $v(\Psi_{\mathcal{M}}\circ F)$ is defined. Thus we will deal with the following equation
\begin{eqnarray}\label{E4-3R}
\overline{\square}_gF-\frac{d\mu_t}{d\mu}\widetilde{v}\frac{\rho}{\widetilde{vol}(F)}\widetilde{\nu}&-&\Psi_{\Omega_t^{\bot}}D_AD_B\Psi_{\mathcal{M}}(F)(\partial_tF^A\partial_tF^B-\frac{d\mu_t}{d\mu}\widetilde{v}g^{ij}\partial_iF^A\partial_jF^B)\nonumber\\
&-&\Psi_{\Omega_t^{\bot}}D_f\frac{d\mu_t}{d\mu}\widetilde{v}\widetilde{\varphi}(f^{-1}\circ F-\nabla(\frac{|F|^2}{2}))\widetilde{\nu}=0,
\end{eqnarray}
where $\widetilde{v}=\Psi_{\Omega_t^{\bot}}v(\Psi_{\mathcal{M}}\circ F)$, $\widetilde{\nu}=\Psi_{\Omega_t^{\bot}}\nu(\Psi_{\mathcal{M}}\circ F)$ and $\widetilde{\varphi}=\Psi_{\Omega_t^{\bot}}\varphi(\Psi_{\mathcal{M}}\circ F)$.

The following result shows that if $F$ maps to $f(\mathcal{M})$ initially and the initial velocity is tangent to $f(\mathcal{M})$, then $F$ maps to $f(\mathcal{M})$ for all time. The proof is the same with Lemma 4.3 in \cite{Notz}, so we omit it.
\begin{lemma}
Assume that equation (\ref{E4-3R}) with initial data $F(0,x)\in f(\mathcal{M})$ and $\partial_tF(0,x)\in\textbf{T}_{F(0,x)}f(\mathcal{M})$ has a smooth solution $F$ on $[0,T]\times\Sigma$. Then $F(t,x)\in f(\mathcal{M})$ and $f^{-1}\circ F$ solves equation (\ref{E1-1}) with initial data $f^{-1}\circ F(0)$ and $Df^{-1}(\partial_tF(0))$ on $[0,T]\times\Sigma$.
\end{lemma}
In what follows, one can use the same process of dealing with the Euclidean case to prove the existence of solution for equation (\ref{E4-3R}). Here we omit the details of proof.\\





\bibliographystyle{elsarticle-num}







\end{document}